\documentclass[12pt]{amsart} 
\usepackage{amsmath}
\usepackage{amsfonts,amssymb}
\usepackage{epsfig,color}
\theoremstyle{plain} %
\newtheorem*{theorem*}{\indent\bf Theorem}
\newtheorem{theorem}{\indent\bf Theorem}

\newtheorem{corollary}{\indent\bf Corollary}
\newtheorem{proposition}{\indent\bf Proposition}

\theoremstyle{definition} %
\newtheorem{definition}{\indent\bf Definition}
\newtheorem{remark}{\indent\bf Remark}

\newtheorem{conjecture}{\indent\bf Conjecture}
\newcommand{\ba}{\bvec a}
\newcommand{\bb}{\bvec b}
\newcommand{\be}{\bvec e}
\newcommand{\boldf}{\bvec f}
\newcommand{\bw}{\bvec w}
\newcommand{\bu}{\bvec u}
\newcommand{\bv}{\bvec v}
\newcommand{\bx}{\bvec x}
\newcommand{\by}{\bvec y}

\newcommand{\balpha}{\bvec \alpha}
\newcommand{\bbeta}{\bvec \beta}
\newcommand{\brho}{\bvec \rho}
\newcommand{\bvarphi}{\bvec \varphi}
\newcommand{\bkappa}{\bvec \kappa}

\newcommand{\tr}{\operatorname{tr}}
\newcommand{\bvec}[1]{\mbox{\boldmath $#1$}}
\newcommand{\End}{\operatorname{End}}
\begin{document}
\title[From colored Jones  to  logarithmic invariants]{{From colored Jones invariants to  logarithmic invariants}}
%
\author[Jun Murakami]{Jun Murakami} 
\address{Department of Mathematics, Faculty of Science and Engineering, Waseda University, 3-4-1 Ohkubo, Shinjuku-ku,
Tokyo 169-8555,  JAPAN.}
\email{murakami@waseda.jp}
\date{} %
%
%
\thanks{
This research is partially supported by the Grant-in-Aid for Scientific Research (B) (25287014),  Exploratory Research (25610022) of Japan Society for the Promotion of Science, 
and the Erwin Schr\"{o}dinger Institute for Mathematical Physics (ESI) in Vienna.  
}
\maketitle
\begin{abstract}
In this work, we give a formula for the logarithmic invariant of knots  in terms of certain derivatives of the colored Jones invariant.  
This invariant is related to the logarithmic conformal field theory, and was defined by using the centers in the radical of the restricted quantum group at root of unity.
A relation between logarithmic invariant and the hyperbolic volume of a cone manifold is also investigated.
\end{abstract}

\section*{Introduction}
The logarithmic invariants of knots are introduced by Nagatomo and the author  \cite{MN} by using the centers in the Jacobson radical of the restricted quantum group $\overline{\mathcal U}_q(sl_2)$ at root of unity.   
In this paper, we give a formula for the logarithmic invariant in terms of the colored Jones invariant.
Let $N$ be a positive integer greater than $1$ and let $\xi$ be the $2N$-th root of unity given by $\xi=\exp(\pi\sqrt{-1}/N)$.  
The center of $\overline{\mathcal U}_{\xi}(sl_2)$ is $3N-1$ dimensional, and 
its good basis \begin{equation}
\{\hat\brho_1, \hat\brho_2, \cdots, \hat\brho_{N-1}, \hat\bvarphi_1, \hat\bvarphi_2, \cdots, \hat\bvarphi_{N-1}, 
\hat\bkappa_0, \hat\bkappa_1, \cdots, \hat\bkappa_N
\}
\label{eq:firstbasis}
\end{equation}
is given by \cite[Sect.  5.2]{F}
which behaves well under certain action of $SL(2, {\mathbb Z})$.        
For a knot $L$, 
let $\gamma_s^{(N)}(L)$ be the logarithmic invariant corresponding to  $\hat\bkappa_s$ of the  above basis,   
and let $V_m(L)$ be the  colored Jones invariant corresponding to the 
$m$ dimensional representation of  
${\mathcal U}_q(sl_2)$ at generic  $q$.  
We get the following two formulas to explain the logarithmic invariant $\gamma_s^{(N)}(L)$ by using derivatives of the  colored Jones invariant
$V_m(L)$.  
\medskip
\par\noindent{\bf Theorem}
(in Theorem \ref{th:main}).
\begin{it}
The invariant $\gamma_s^{(N)}(L)$ $(1 \leq s \leq N)$ is given by
\begin{equation}
\begin{aligned}
\gamma_s^{(N)}(L) 
&= 
\frac{\xi}{2N} 
\left.\frac{d}{dq}(q - q^{-1})  \big(V_{s}(L) + V_{2N-s}(L)\big)\right|_{q=\xi}
\\&=
\frac{N}{\pi\, \sqrt{-1}}\, (\xi - \xi^{-1}) 
\left.\frac{d}{d m} {V}_m(L)\right|_{\genfrac{}{}{0pt}{}{m=s}{q = \xi}}. 
\end{aligned}
\label{eq:formula}
\end{equation}
\end{it}
\begin{remark}
The first formula in \eqref{eq:formula} is given by the derivative of $V_m(L)$ with respect to the parameter $q$.  
The second formula is given by the derivative of $V_m(L)$ with respect to the parameter $m$, which is an integer. 
However, we can differentiate $V_m(L)$ with respect to $m$ by using the following universal expression of $V_m(L)$ given by Habiro \cite[Theorem 3.1]{H} (see also \cite{Ma}).
\begin{equation}
V_m(L) 
=
\sum_{i=0}^\infty a_i(L) \, \dfrac{\{m+ i, 2\,i+1\}_q}{\{1\}_q}.  
\label{eq:universal}
\end{equation}
Here $\{n\}_q = q^n - q^{-n}$, $\{n, k\}_q = \prod_{j=0}^{m-1} \{n-j\}_q$  and
the coefficient $a_i(L)$ is a Laurent polynomial in $q$ which does not depend on $m$ (see \cite[Theorem 2.1]{H}).  
For   $\frac{d}{dm}\, V_m(L)$ in \eqref{eq:formula}, $V_m(L)$ is given by \eqref{eq:universal} and is considered to be an infinite sum with the indeterminate $m$. 
The integer $s$ is substituted to $m$ after the differentiation, and  
the sum reduces to a finite sum when  $q$ is specialized to $\xi$.     
\end{remark}
\par
The above theorem suggests some relation between the logarithmic invariant and the hyperbolic volume since relations between the colored Jones invariants and the hyperbolic volume are known for various cases by  \cite{K}, \cite{MM}, \cite{MMOTY}, \cite{G}, \cite{HM} and \cite{MY}.   
Let $L$ be a hyperbolic knot.  
In \cite{K}, Kashaev found a relation between the hyperbolic volume of the knot complement and the series of invariants $\left<L\right>_N$ he constructed.
Kashaev's  invariant turned out to be a specialization of the colored Jones invariant by \cite{MM}, more precisely,   
$\left<L\right>_N = \left.V_{N}(L)\right|_{q=\xi}$. 
Then Kashaev's conjecture is generalized as follows.  
\begin{conjecture}[Complexified volume conjecture \cite{MMOTY}]
Let $L$ be a hyperbolic knot in $S^3$. 
Then
\begin{equation}
\lim_{N\to\infty}\frac{2\, \pi \, \log \left<L\right>_N }{N}
=
\rm{Vol}\left(S^3\setminus L\right) +
\sqrt{-1} \, \rm{CS}\left(S^3 \setminus L\right), 
\label{eq:conjecture}
\end{equation}
where $\rm{Vol}\left(S^3\setminus L\right)$ and 
$\rm{CS}\left(S^3 \setminus L\right)$ are the hyperbolic volume and the Chern-Simons invariant of $S^3 \setminus L$ respectively.  
\end{conjecture}  
%
%
%
There are several generalizations of this conjecture.  
For example, 
if we deform $\xi$ to $\xi^\alpha = \exp(\pi \sqrt{-1}\alpha/N)$ by a complex number $\alpha$ near $1$, a conjecture for the relation between $V_{N}(L)$ at $q = \xi^\alpha$ and the complex volume of  certain deformation of the hyperbolic structure of $S^3\setminus K$ is proposed by \cite{G} and \cite{HM}.   
For the figure-eight knot, this conjecture is proved partially by Murakami-Yokota \cite{MY}.  
\par
Our invariant $\gamma_s^{(N)}(L)$ can be considered as a deformation of $\left<L\right>_N$ since $\left<L\right>_N$ is equal to $\gamma_N^{(N)}(L)$.  
Changing the parameter $N$ to $s$ can be considered as a deformation (not continuous but discrete) of the weight parameter $\lambda$ instead of the deformation of the parameter $q$.  
Comparing with the deformations in \cite{G}, \cite{HM}, \cite{MY},  we propose the following conjecture.  
\begin{conjecture}[Volume conjecture for the logarithmic invariant]
Let $L$ be a hyperbolic knot and
let $M_\alpha$ be the cone manifold along the singularity set $L$ with the cone angle $\alpha$ with $0 \leq \alpha \leq \pi$.  
Let $s_N^{\alpha}$ be a sequence of integers such that 
$\lim_{N\to\infty}{s_N^{}}/{N} = 1-{\alpha}/{2\, \pi}$.  
If $M_\alpha$ is a hyperbolic manifold, then
\begin{equation}
\lim_{N\to\infty}\frac{2\, \pi \log \gamma_{s_N^{\alpha}}^{(N)}(L)}{N} = \rm{Vol}(M_\alpha) + \sqrt{-1}\,\rm{CS}(M_\alpha).
\label{eq:logconjecture}
\end{equation}
\end{conjecture}
\par
For the figure-eight knot, we prove this conjecture for $\alpha$ satisfying $0 \leq \alpha <  {\pi}/{3}$, and check numerically for all $\alpha$.  
\par
This paper is organized as follows.  
In Sect. 2, we recall the construction of the colored Jones invariant.  
In Sect. 3, we recall the restricted quantum groups, their representations and their centers.  
These materials are explained in \cite{F}.  
In Sect. 4, we discuss about the logarithmic invariant of knots.  
For a knot $L$, there is a tangle $T_L$ corresponding to $L$, and by passing through the universal invariant by Lawrence \cite{L} and Ohtsuki \cite{O}, we get a center $z(T_L)$ of $\overline{\mathcal U}_{\xi}(sl_2)$, which is an invariant of $L$.  
We introduce a representation of ${\mathcal U}_q(sl_2)$ for generic $q$, which coincides with a projective representation of the restricted quantum group $\overline{\mathcal U}_{\xi}(sl_2)$ when $q$ is specialized to $\xi$.  
Then, by applying this specialization to the  colored Jones invariant $V_m(L)$,  we get a formula to explain the logarithmic invariant in terms of $V_m(L)$.  
Moreover, since the invariant $z(T_L)$ is a linear combination of the  basis in \eqref{eq:firstbasis}, these coefficients are again invariants of  $L$,  
and they are expressed  in terms of $V_m(L)$.   
In Sect. 5, we investigate the relation between the logarithmic invariant of the figure-eight knot $K_{4_1}$ and the hyperbolic volume of a cone manifold along  $K_{4_1}$.    
\par\noindent{\bf Acknowladgement.}
I would like to thank Gregor Masbaum for valuable discussion.   
\section{Colored Jones invariant}
\subsection{Notations}
Let $q$ be a parameter, $\xi = \exp(\pi\sqrt{-1}/N)$ be
the primitive $2N$-th root of unity, and
 we use the following notations.   
$$
\begin{aligned}
\{n\}_q &= q^n - q^{-n}, \quad
\{n, m\}_q = \prod_{k=0}^{m-1} \{n-k\}_q, \quad
\{n\}_q! = \{n, n\}_q, 
\\
\{n\} &= \{n\}_{\xi}, \quad
\{n\}! = \{n\}_{\xi}! , \quad
[n] =[n]_{\xi}, \quad
[n]! = [n]_{\xi}!, \quad
\{n\}_+ = \xi + \xi^{-1},
\\
[n]_q &= \frac{\{n\}_q}{\{1\}_q}, \quad
[n]_q! = \prod_{k=1}^n[k]_q, \quad
\left[\begin{matrix} n \\ k\end{matrix}\right]_q
= \frac{[n]_q!}{[k]_q! \,[n-k]_q!},
\quad
\left[\begin{matrix} n \\ k\end{matrix}\right]
=
\left[\begin{matrix} n \\ k\end{matrix}\right]_{\xi}.  
\end{aligned}
$$
\subsection{Quantum group ${\mathcal U}_{q}(sl_2)$}
Let ${\mathcal U}_q(sl_2)$ be the quantum group defined by
$$
\begin{aligned}
{\mathcal U}_q(sl_2)
=
\left<K,\ E,\ F  \right.\ \mid&\ \
K\, E\, K^{-1} = q^{2} \, E, \ \ 
K\, F\, K^{-1} = q^{-2} \, F,
\\
&\ \left.
E\, F - F\, E = 
\frac{K - K^{-1}}{q-q^{-1}}\right>.  
\end{aligned}
$$
The Hopf algebra structure of ${\mathcal U}_q(sl_2)$ is given by
$$
\begin{aligned}
\Delta(K) &= K \otimes K, \ \ 
\Delta(E) = 1 \otimes E + E \otimes K, \ \ 
\Delta(F) = K^{-1} \otimes F + F \otimes 1, 
\\
\epsilon(K) &=1, \ \ 
\epsilon(E) = \epsilon(F) = 0, 
\\
S(K) &= K^{-1}, \ \ 
S(E) = -E\, K^{-1}, \ \ 
S(F) = - K\, F, 
\end{aligned}
$$
where $\Delta$ is the coproduct, $\epsilon$ is the count and $S$ is the antipode.  
The universal $R$-matrix of ${\mathcal U}_q(sl_2)$ is given by
\begin{equation}
R = 
q^{\frac{1}{2}H\otimes H} \, 
\sum_{n=0}^\infty
\frac{\{1\}_q^{2n}}{\{n \}_q!} q^{\frac{n(n-1)}{2}}\,(E^n \otimes F^n),  
\label{eq:r-matrix}
\end{equation}
where $H$ is an element such that $q^H = K$.  
\subsection{Irreducible representations of ${\mathcal U}_q(sl_2)$}
Let $W_{m}$ be the highest weight representation of the  quantum group ${\mathcal U}_q(sl_2)$ given by the following basis and actions.  
Let $\boldf_0$, $\boldf_1$, $\cdots$, $\boldf_{m-1}$ be the weight basis of $W_{m}$ and the actions of $E$, $F$, $K$ are given by 
$$
E \, \boldf_i = [i]_q \, \boldf_{i-1}, \quad 
F \, \boldf_i = [m-1-i]_q \, \boldf_{i+1},  \quad
K \, \boldf_i = q^{m-1-2i} \, \boldf_i.
$$
Then $W_m$ is irreducible  if $q$ is generic.  
Let $\rho_m : {\mathcal U}_q(sl_2) \to \End(W_m)$ be the algebra homomorphism defined by the above actions.  
\subsection{Colored Jones invariants}
Here we explain the colored Jones invariants briefly.
For detail, see \cite{RT}.  
For a knot $L$ w, let $b_L$ be a braid whose closure is  isotopic to $L$ as a framed knot.  
Let $n$ be the number of strings of $b_L$.  
By assigning universal $R$ matrix at each clossing of a braid, a represntation of the braid group $B_n$  is defined on $W_m^{\otimes n}$, i.e. we have a homomorphism $\rho_m^{(n)} : B_n \to \End(W_m^{\otimes n})$.  
The colored Jones invariant $V_m(L)$ of $L$ is given by the quantum trace of 
$\rho_m^{(n)}(b_L)$.  
More precisely, 
$
V_m(L) = \tr\big(\rho^m(K)^{\otimes n} \, \rho_m^{(n)}(b_L)\big)
$.
Later, we use the {\it normalized colored Jones invariant} $\widetilde V_m(L)$ which is defined by
$
\widetilde V_m(L) = {V_m(L)}/{[m]}  
$.
%
%
\subsection{Tangle invariant}
The normalized colored Jones invariant $\widetilde V_m(L)$ can be interpreted as an invariant of a $(1,1)$-tangle $T_L$ whose closure is the knot $L$.  
\par
Let $V$ be a $d$ dimensional representation of ${\mathcal U}_{\xi}(sl_2)$  with basis 
$\{\be_0, \be_1, \cdots, \be_{d-1}\}$ and $\rho_V : {\mathcal U}_{\xi}(sl_2) \to \End(V)$ be the corresponding algebra homomorphism. 
Let $b_L \in B_n$ as before, then there is a homomorphism $\rho_V^{(n)} : B_n \to \End(V^{\otimes n})$ denifed by the universal $R$ matrix.  
Let $T_L$ be a (1,1)-tangle obtained from $b_L$ by taking the closure of the right $(n-1)$ strings, then the closure of $T_L$ is $L$.  
On the other hand, by taking the partial trace of $\rho_V^{(n)}(b_L)$ corresponding to the right $(n-1)$ components of $V^{\otimes n}$, we get a operator in $\End(V)$.   
Here the partial trace $\widetilde\tr$ is given as follows.   
$$
\widetilde\tr\big(\rho_V^{(n)}(b_L)\big)_{i_1}^{j_1} = 
\sum_{i_2, i_3, \cdots, i_n=0}^{d-1}
\Big((id \otimes \rho_V^{}(K)^{\otimes (n-1)})\,\rho_V^{(n)}(b_L)
\Big)_{i_1, i_2, \cdots, i_n}^{j_1, i_2, \cdots, i_n}.  
$$
The operator $\widetilde\tr(\rho_V^{(n)}(b_L))$ is an isotopy invariant of the knot $L$, and 
if $V=W_m$,  $\widetilde\tr(\rho_m^{(n)}(b_L))$ is
a scalar matrix since $W_m$ is irreducible.  
This scalar is equal to $\widetilde V_m(L)$.  
\section{Restricted quantum group $\overline{\mathcal U}_{\xi}(sl_2)$}
We introduce the resetricted quantum group and its representations.  
\subsection{Restricted quantum group $\overline{\mathcal U}_{\xi}(sl_2)$}
\begin{definition}
The {\it restricted quantum group} $\overline{\mathcal U}_{\xi}(sl_2)$ is given by 
$$
\overline{\mathcal U}_{\xi}(sl_2)
=
{\mathcal U}_{\xi}(sl_2)/(E^N,\ F^N,\  K^{2N} - 1),
$$
i.e.~$\overline{\mathcal U}_{\xi}(sl_2)$ is defined from  ${\mathcal U}_{\xi}(sl_2)$ by adding new relations $E^N = F^N = 0$ and $K^{2N} = 1$.  
\end{definition}
The  $R$ matrix of \,$\overline{\mathcal U}_{\xi}(sl_2)$ is given by
\begin{equation}
R = 
(\xi^{\frac{1}{2}})^{H\otimes H} \, 
\sum_{n=0}^{N-1}
\frac{\{1\}^{2n}}{\{n \}!} \xi^{\frac{n(n-1)}{2}}\,(E^n \otimes F^n).    
\label{eq:r-matrixlog}
\end{equation}
Here $\xi^{\frac{1}{2}} = \exp\left({\pi\sqrt{-1}}/{2N}\right)$,
$H$ is given by $\xi^{\,H} = K$ and  satisfies
$$
H \, E - E \, H = 2\, E, \quad 
H \, F - F \, H = -2\, F.
$$ 
Moreover, for every $\overline{\mathcal U}_{\xi}(sl_2)$-module $V$, 
if $K\, v = v$ for $v \in V$, then we assume $H\, v = 0$.   
With the above assumptions, the representation of the $R$ matrix on the tensor representation of two projective modules explained in the next subsection is uniquely determined, and coincides with the representation of the universal $R$-matrix of $\overline{\mathcal U}_{\xi}(sl_2)$ given by Drinfeld's quantum double constriction in  \cite{F}.    
\subsection{Projective modules of $\overline{\mathcal U}_{\xi}(sl_2)$}
We first explain irreducible representations of  $\overline{\mathcal U}_\xi(sl_2)$.  
Let $U_s^{\pm}$ be the $s$-dimensional irreducible representations of $\overline{\mathcal U}_{\xi}(sl_2)$ labeled by  $1 \leq s \leq N$.  
The module  $U_s^{\pm}$ is  spanned by elements $\bu_n^{\pm}$ for $0 \leq n \leq s-1$, where the action of $\overline{\mathcal U}_{\xi}(sl_2)$  is given by
$$
\begin{aligned}
K \, \bu_n^{\pm} 
&= 
\pm \xi^{s-1-2n} \, \bu_n^{\pm}, 
&0 \leq n \leq s-1,\qquad&
\\
E \, \bu_n^{\pm} &= \pm [n][s-n] \, \bu_{n-1}^{\pm}, 
\quad  &1 \leq n \leq s-1,
\qquad&\qquad
E \, \bu_0^{\pm} = 0, 
\\
F \, \bu_n^{\pm} &= \bu_{n+1}^{\pm}, 
\quad
&0 \leq n \leq s-2,
\qquad&\qquad
F \, \bu_{s-1}^{\pm} = 0 .  
\end{aligned}
$$
Especially, $U_1^{+}$ is the trivial module for which $K$ acts by 1 and  $E$, $F$ act by 0.  
The weights (eigenvalues of $K$) occurring in $U_s^{+}$ are
$
 \xi^{s-1}$, $\xi^{s-3}$, $\cdots$, $\xi^{-s+1}$,
and the  weights occurring in $U_{N-s}^{-}$ are
$
-\xi^{N-s-1}$, $-\xi^{N-s-3}$, $\cdots$, $-\xi^{-N+s+1}$.  
\par
Let  $V_s^{\pm}$ $(1 \leq s \leq N)$ be the $N$ dimensional representation with highest-weight 
$\xi^{s-1}$ spanned by elements $\bv_n^{\pm}$ for $0 \leq n \leq N-1$, where the action of $\overline{\mathcal U}_{\xi}(sl_2)$  is given by 
$$
\begin{aligned}
K \, \bv_n^{\pm} 
&= 
\pm\xi^{s-1-2n} \, \bv_n^{\pm}, 
\quad& 0 \leq n \leq N-1, \qquad&
\\
E \, \bv_n^{\pm} &= \pm [n][s-n] \, \bv_{n-1}^{\pm}, 
\quad  &1 \leq n \leq N-1,
\qquad&\qquad
E \, \bv_0^{\pm} = 0, 
\\
F \, \bv_n^{\pm} &= \bv_{n+1}^{\pm}, 
\quad&
0 \leq n \leq N-2,
\qquad&\qquad
F \, \bv_{N-1}^{\pm} = 0 .  
\end{aligned}
$$
Note that   
$V_N^\pm = U_N^\pm$.   
For $1 \leq s \leq N-1$,  $V_s^{\pm}$ satisfies the exact sequence
$$
0 \longrightarrow  U_{N-s}^{\mp} \longrightarrow
V_s^{\pm} \longrightarrow U_s^{\pm} \longrightarrow 0,   
$$
and there are projective modules 
$P_s^{\pm}$ satisfying the following exact sequence.  
$$
0 \longrightarrow  V_{N-s}^{\mp} \longrightarrow
P_s^{\pm} \longrightarrow V_s^{\pm} \longrightarrow 0.
$$
The module ${\mathcal P}_s^{+}$ has a basis
$
\{\bx_j^{+},\ \by_j^{+}\}_{0 \leq j \leq N-s-1}
\cup
\{\ba_n^{+}, \ \bb_n^{+}\}_{0 \leq n \leq s-1},   
$
and the actions of $\overline{\mathcal U}_{\xi}(sl_2)$ is determined by
$$
\begin{aligned}
&K\, \bx_j^{+} = \xi^{2N-s-1-2j} \, \bx_j^{+}, 
\ \ \quad
K \, \by_j^{+}=\xi^{-s-1-2j}\, \by_j^{+},
\quad
0 \leq j \leq N-s-1,
\\
&K \, \ba_n^{+} = \xi^{s-1-2n} \, \ba_n^{+},
\qquad\ \quad
K \, \bb_n^{+} = \xi^{s-1-2n} \, \bb_n^{+}, 
\quad\ \ \,
0 \leq n \leq s-1,
\end{aligned}
$$
\begin{equation}
\begin{aligned}
&\begin{matrix}
E \, \bx_j^{+} =
-[j][N-s-j] \, \bx_{j-1}^{+},\\
\qquad\qquad
 0 \leq j \leq N-s-1,
 \end{matrix}
\quad
E\, \by_j^{+} = 
\begin{cases}
-[j][N-s-j] \, \by_{j-1}^{+}, \\
\qquad\qquad1 \leq k \leq N-s-1,
\\
\ba_{s-1}^{+}, \qquad j=0,
\end{cases}
\\
&\!\!\begin{matrix}E \, \ba_n^{+} = 
[n][s-n]  \ba_{n-1}^{+},  \\
\qquad\qquad 0 \leq n \leq s-1,
\end{matrix}
\ \
E \, \bb_n^{+} =
\begin{cases}
[n][s-n] \bb_{n-1}^{+} + \ba_{n-1}^{+}, & \!\!1 \leq n \leq s-1,
\\
\bx_{N-s-1}^{+}, & n=0,
\end{cases}
\\
&F\, \bx_j^{+} = 
\begin{cases}
\bx_{j+1}^{+}, & 0 \leq j \leq N-s-2,
\\
\ba_0^{+}, & j = N - s - 1,
\end{cases}
\ \ 
F\, \by_j^{+} =
\by_{j+1}^{+}, \  0 \leq j \leq N-s-2,
\\
&
F \, \ba_n^{+} = \ba_{n+1}^{+},
\quad 0 \leq n \leq s-1,
\qquad\quad
F\, \bb_n^{+} =
\begin{cases}
\bb_{n+1}^{+}, & 0 \leq n \leq s-2,
\\
\by_0^{+}, & n = s-1.  
\end{cases}
\end{aligned}
\label{eq:pp}
\end{equation}
Here we assume that 
$\bx_{-1}^+=\ba_{-1}^+=\by_{N-1}^+ = \ba_s^+=0$. 
\par  
The  module ${\mathcal P}_{N-s}^{-}$  has a basis
$
\{\bx_j^{-},\ \by_j^{-}\}_{0 \leq j \leq N-s-1}
\cup
\{\ba_n^{-}, \ \bb_n^{-}\}_{0 \leq n \leq s-1},   
$ 
and the action of $\overline{\mathcal U}_{\xi}(sl_2)$ is determined by
\begin{equation}
\begin{aligned}
&K \, \bx_j^{-} =  \xi^{-s-1-2j} \, \bx_j^{-}, 
\qquad
K \, \by_j^{-} =  \xi^{-s-1-2j}\, \by_j^{-},
\quad
0 \leq j \leq N-s-1,
\\
&K \, \ba_n^{-} = \xi^{s-1-2n} \, \ba_n^{-},
\qquad\ 
K \, \bb_n^{-} = \xi^{-2N+s-1-2n} \, \bb_n^{-}, 
\quad
0 \leq n \leq s-1,
\\
&\begin{matrix}
E \, \bx_j^{-} =
-[j][N\!-\!s\!-\!j] \, \bx_{j-1}^{-}, \\
\qquad\qquad 0 \leq k \leq N-s-1,
\end{matrix}
\quad
E\, \by_j^{-} = 
\begin{cases}
-[j][N\!-\!s\!-\!j] \, \by_{j-1}^{-} + \bx_{j-1}^{-}, \\
\qquad\qquad {1 \leq j \leq N\!-\!s\!-\!1,}
\\
\ba_{s-1}^{-}, \qquad j=0,
\end{cases}
\\
&\begin{matrix}
E \, \ba_n^{-} = 
[n][s-n] \, \ba_{n-1}^{-}, \qquad\\ 
\qquad\qquad\quad 0 \leq n \leq s-1,
\end{matrix}
\ \ \ 
E \, \bb_n^{-} =
\begin{cases}
[n][s-n]\, \bb_{n-1}^{-}, & 1 \leq n \leq s-1,
\\
\bx_{N-s-1}^{-}, & n=0,
\end{cases}
\\
&\begin{matrix}
F\, \bx_j^{-} =
\bx_{j+1}^{-}, \qquad \qquad\qquad\quad
\\
0 \leq j \leq N-s-2,
\end{matrix}
F\, \by_j^{-} = 
\begin{cases}
\by_{j+1}^{-}, & 0 \leq j \leq N-s-2,
\\
\bb_0^-, & j = N - s - 1,
\end{cases}
\\
&F\, \ba_n^{-} =
\begin{cases}
\ba_{n+1}^{-}, & 0 \leq n \leq s-2,
\\
\bx_0^{-}, & n = s-1.  
\end{cases}
\qquad
F \, \bb_n^{-} = \bb_{n+1}^{-},
\qquad 0 \leq n \leq s-1.
\end{aligned}
\label{eq:pm}
\end{equation}
Here we assume that $\bx_{-1}^- = \ba_{-1}^- = \bx_{N-s}^- = \bb_s^- = 0$.  
\subsection{Centers of $\overline{\mathcal U}_{\xi}(sl_2)$}
The center of $\overline{\mathcal U}_{\xi}(sl_2)$ is investigated in \cite{F}.  
\smallskip
\par\noindent
\begin{proposition}[\cite{F}, 4.4.4.]
The dimension of the center $Z$ of  $\overline{\mathcal U}_{\xi}(sl_2)$ is $3N-1$.  
Its commutative algebra structure is described as follows.  
There are two special central idempotents $\be_0$ and $\be_N$, other central idempotents $\be_s$ $(1 \leq s \leq N-1)$, centers in the radical $\bw_s^{\pm}$ $(1 \leq s \leq N-1)$, and they satisfy the following commutation relation.  
$$
\begin{aligned}
\be_s \, \be_t= \delta_{s, t}\, \be_s,
\qquad\qquad\quad\ \,
& 0 \leq s, t \leq N,
\\
\be_s \, \bw_t^{\pm}= \delta_{s,t} \, \bw_t^{\pm},\qquad
\qquad\ 
& 0 \leq s \leq N, \ 1 \leq t \leq N-1,
\\
\bw_s^{\pm} \, \bw_t^{\pm} = \bw_s^{\pm} \, \bw_t^{\mp} = 0,
\qquad
& 1 \leq s,\ t \leq N-1.  
\end{aligned}
$$
\end{proposition}
\par
The center $\be_N$ acts as identity on $U_N^{+}$  and  as 0 on the other modules.  
$\be_0$ acts as identity on $U_N^{-}$  and  as 0 on the other modules.  
$\be_s$ acts as identity on $P_s^{+}$, $P_{N- s}^{-}$  and  as 0 on the other modules.  
The center $\bw_s^{+}$ acts as ${\mathcal P}_s^{+}$ by
$\bw_s^{+} \, \bb_n^{+}  = \ba_n^{+}$, 
$\bw_s^+ \, \ba_n^{+} = \bw_s^{+} \, \bx_k^{+} = \bw_s^{+} \, \by_k^{+} = 0$, and
acts on the other modules as 0.  
Similarly, $\bw_s^{-}$ acts on ${\mathcal P}_{N-s}^{-}$ by
$\bw_s^{-} \, \by_k^{-}  = \bx_k^{-}$, 
$\bw_s^{-} \, \bx_k^{-} = \bw_s^{-} \, \ba_n^{-} = \bw_s^{-} \, \bb_n^{-} = 0$, and
acts on the other modules as 0.  
\par
According to \cite{F}, the basis \eqref{eq:firstbasis} is expressed by $\be_s$ and $\bw_s^\pm$ as follows.   
$$
\begin{aligned}
\hat\brho_s 
&= 
(-1)^{N+s} \, \dfrac{1}{N\,(q^{s} - q^{-s})} \, 
\left(\be_s - \dfrac{q^{s} + q^{-s}}{[s]^2} \, (\bw_s^+ +\bw_s^-)
\right)
\quad{(1 \leq s \leq N-1)},
\\
 \hat\bvarphi_s
&=
\dfrac{1}{[s]^2} \, \left(
\dfrac{N-s}{N} \, \bw_s^+ - \dfrac{s}{N} \, \bw_s^-
\right)
\qquad\qquad\qquad\qquad\qquad\quad\ {(1 \leq s \leq N-1)}.
\\
\hat\bkappa_0
&=
\be_0,
\qquad
\hat\bkappa_s
=
\dfrac{1}{[s]^2} \, \left(\bw_s^+ + \bw_s^-\right)
\quad
{(1 \leq s \leq N-1)},
\qquad
\hat\bkappa_N
=
-\be_N.
\end{aligned}
$$
Let $z$ be a center of $\overline {\mathcal U}_q(sl_2)$ given by
\begin{equation}
z = 
a_0 \,  \be_0 + a_N  \, \be_N + 
\sum_{s=1}^{N-1}
\left(
a_s  \, \be_s + b_s^+  \, \bw_s^+ + b_s^-  \, \bw_s^-
\right). 
\label{eq:ab} 
\end{equation}
Then $z$ can be also expressed by the good basis $\hat \bkappa_s$, $\hat\brho$, $\hat\bvarphi$ by  
\begin{equation}
z = 
\sum_{s=1}^{N-1}
\alpha_s^{(N)} \, \hat \brho_s 
+
\sum_{s=1}^{N-1}
\beta_s^{(N)} \, \hat \bvarphi_s 
+
\sum_{s=0}^N
\gamma_s^{(N)} \, \hat\bkappa_s, 
\label{eq:gamma}
\end{equation}
where
\begin{equation}
\begin{aligned}
\alpha_s^{(N)} &= (-1)^{N+s} \, (q^s - q^{-s}) \, N \, a_s, 
\quad
\beta_s^{(N)} =  [s]^2 \, (b_s^+ - b_s^-),
\\
\gamma_s^{(N)} &= [s]^2 \,\left(
\dfrac{s}{N}b_s^+ + \dfrac{N-s}{N} \, b_s^-\right) +
(q^s + q^{-s}) \, 
a_s, 
\qquad(1 \leq s \leq N-1)
\\
\gamma_0^{(N)} &=a_0, 
\qquad\qquad\qquad\ 
\gamma_N^{(N)} = -a_N.  
\end{aligned}
\label{eq:goodbasis}
\end{equation}
\section{Logarithmic invariants of knots}
\subsection{Logarithmic invariants}
Let $L$ be a knot with framing $0$, $T_L$ be a tangle obtained from $L$ and $z(T_L)$ be the center corresponding to the universal invariant constructed by Lawrence and Ohtsuki, where we assign the $R$ matrix given by \eqref{eq:r-matrix} and $K^{\pm1}$ to the maximal and the minimal points as in Figure \ref{figure:diagram}.    
In \cite{MN}, $K^{N\pm1}$ is assigned instead of $K^{\pm1}$, and so the  invariant defined here and that in \cite{MN} is different by the sign $(-1)^{(m-1) f}$ where $f$ is the framing of the knot.  
So these invariants coincide for an unframed knot.  
\par
From \eqref{eq:ab} and \eqref{eq:gamma}, we define $a_k(L)$, $b_k^\pm(L)$, $\alpha_k^{(N)}(L)$, $\beta_k^{(N)}(L)$, $\gamma_k^{(N)}(L)$  as follows.  
$$
\begin{aligned}
z(T_L) & = 
a_0(L) \,  \be_0 + a_N(L)  \, \be_N + 
\sum_{s=1}^{N-1}
\Big(
a_s(L)  \, \be_s + b_s^+(L)  \, \bw_s^+ + b_s^-(L)  \, \bw_s^-
\Big) 
\\
&= 
\sum_{s=1}^{N-1}
\alpha_s^{(N)}(L) \, \hat \brho_s 
+
\sum_{s=1}^{N-1}
\beta_s^{(N)}(L) \, \hat \bvarphi_s 
+
\sum_{s=0}^N
\gamma_s^{(N)}(L) \, \hat\bkappa_s.
\end{aligned}
$$
The purpose of this section is to express the above coefficients in terms of the colored Jones inariant.  
We first consider $b_s^\pm(L)$.    
\begin{proposition}
\label{prop:b}
Let $L$ be a knot.  
Then we have
\begin{equation}
\begin{aligned}
b_{s}^+(L)
&=
\dfrac{\xi }{2\,N\, [s]}\, 
\left.
\frac{d}{dq} \, \{1\}_q\left(\dfrac{V_{s}(L)}{[s]_q} - 
\dfrac{V_{2N-s}(L)}{[2N-s]_q} \right)\right|_{q=\xi},
\\
b_{s}^-(L) 
&=
\dfrac{\xi}{4\,N\, [s]}\, 
\left.
\frac{d}{dq} \, \{1\}_q\left(\dfrac{V_{2N+s}(L)}{[2N+s]_q} - 
\dfrac{V_{2N-s}(L)}{[2N-s]_q} \right)\right|_{q=\xi}.
\end{aligned}
\label{eq:differential}
\end{equation}
\end{proposition}
The proof of this propositin is given in Sect. \ref{ss:proof}.  
\begin{figure}[htb]
$$
\epsfig{file=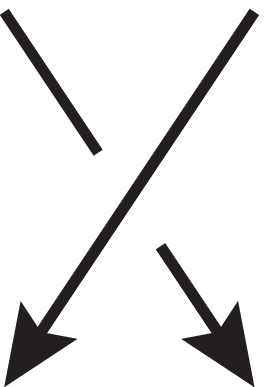, scale=0.5}\ \ 
\raisebox{0.9cm}{$\longrightarrow$}\ \ 
\epsfig{file=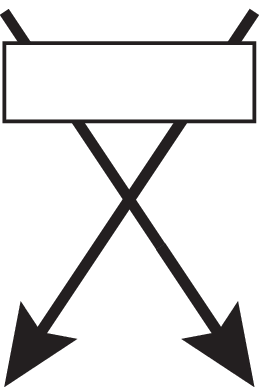, scale=0.5}
\raisebox{1.4cm}{\hspace{-9mm}$R$}
\hspace{8mm}
,
\qquad\qquad
\epsfig{file=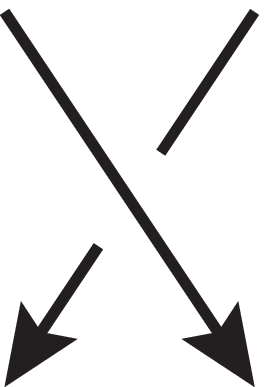, scale=0.5}\ \ 
\raisebox{0.9cm}{$\longrightarrow$}\ \ 
\epsfig{file=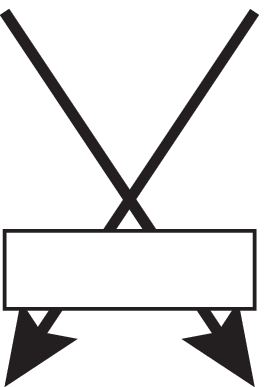, scale=0.5}
\raisebox{0.5cm}{\hspace{-9mm}$\scriptstyle R^{-1}$}
\hspace{6mm}
,
$$
\vspace{2mm}
$$
\epsfig{file=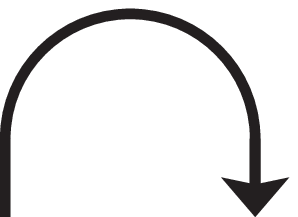, scale=0.6} \ \ 
\raisebox{0.5cm}{$\longrightarrow$}\ \ 
\epsfig{file=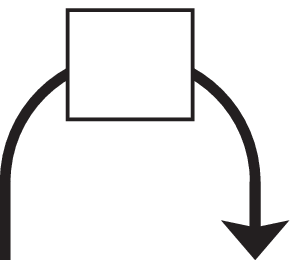, scale=0.6}
\raisebox{1.1cm}{\hspace{-13mm}$\scriptstyle K^{-1}$}
\hspace{9mm},
\qquad
\epsfig{file=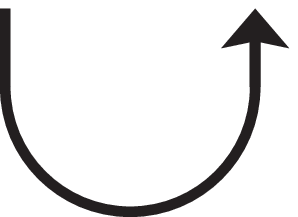, scale=0.6} \ \ 
\raisebox{0.5cm}{$\longrightarrow$}\ \ 
\epsfig{file=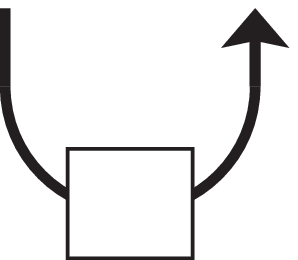, scale=0.6}
\raisebox{0.2cm}{\hspace{-12mm}$ K$}
\hspace{10mm}.
$$
\caption{Assignment of the $R$ matrix and $K^{\pm1}$}
\label{figure:diagram}
\end{figure}
\subsection{Modified representations of ${\mathcal U}_q(sl_2)$}
Let $W_{m}$ be the highest weight representation of the  quantum group ${\mathcal U}_q(sl_2)$ given in Sect. 1.3.  
For an integer $m$ ($1 \leq m \leq N-1$),  we introduce a $2N$ dimensional representation ${\mathcal Y}_m^+$ which is isomorphic to  $W_{2N-m} \oplus W_{m}$.  
The basis of ${\mathcal Y}_m^+$ is $\balpha_0^+$, $\balpha_1^+$, $\cdots$, $\balpha_{2N-m-1}^+$, $\bbeta_0^+$, $\bbeta_1^+$, $\cdots$, $\bbeta_{m-1}^+$, and the actions of $E$, $F$, $K \in {\mathcal U}_q(sl_2)$ are given by
$$
\begin{aligned}
E \, \balpha_i^+ &= 
\begin{cases}
[i]_q \, \balpha_{i-1}^+ \qquad\qquad\qquad\qquad\qquad\qquad 
\text{if $i \leq N-m$ or $i \geq N+1$} ,
\\
[i]_q \, \balpha_{i-1}^+ + \left[\begin{matrix}
2N-m-i-1 \\[-5pt] N-i\end{matrix}\right]_q \, \bbeta_{m-N+i-1}^+ 
\end{cases}
\\ 
&\qquad\qquad\qquad\qquad\qquad\qquad\qquad\qquad\qquad\qquad
\text{if $N-m+1 \leq i \leq N$},
\\
E\, \bbeta_i^+ &= [i]_q \, \bbeta_{i-1}^+,
\\
F \, \balpha_i^+ &=\begin{cases}
[2N-m-i-1]_q \, \balpha_{i+1}^+ & \text{if $i \neq N-m - 1$},
\\
[N]_q \, \balpha_{i+1}^++ \left[\begin{matrix}
N-1 \\[-5pt] m-1\end{matrix}\right]_q \, \bbeta_{0}^+ & 
\text{if $i = N-m -1$},
\end{cases}
\\
F\, \bbeta_i^+ &= [m-i-1]_q \, \bbeta_{i+1}^+, 
\\
K \, \balpha_i^+ &= q^{2N-m-1-2i} \, \balpha_i^+, 
\qquad\qquad\qquad\quad
K \, \bbeta_i^+ = q^{m-1-2 i} \, \bbeta_i^+.  
\end{aligned}
$$
\begin{proposition}
If $q$ is specialized to $\xi$, then ${\mathcal Y}_m^+$ is isomorphic to the projective module ${\mathcal P}_{m}^+$ given by \eqref{eq:pp}.  
\end{proposition}
\begin{proof}
We compare  the actions of $\overline{\mathcal U}_{\xi}(sl_2)$ on ${\mathcal Y}_m^+$ and ${\mathcal P}_{m}^+$.  
Let $f$ be a linear map defined by
\begin{equation}
\begin{aligned}
f(\bx_k^+) &= \frac{(-1)^{N-m-1-k}\,[m]}{[N-m-1-k]!} \,  \balpha_k^+, \quad
f(\by_k^+) = \frac{(-1)^{k}\, [N-1]!}{[N-m-1-k]!}\,  \balpha_{N+k}^+
\\
&\qquad\qquad\qquad\qquad\qquad\qquad\qquad\qquad\qquad
 \text{for $0 \leq k \leq N-m-1$},
\\
f(\ba_k^+) &=\frac{[m]!}{[m-1-k]!} \, \bbeta_k^+,
\quad
f(\bb_k^+) = {[k]!} \, \balpha_{k+N-m}^+
\quad
\text{for $0 \leq k \leq m-1$.}
\end{aligned}
\label{eq:cor1}
\end{equation}
Then a simple computation shows that the actions of $K$, $E$, $F$ on ${\mathcal Y}_m^+$ and $\mathcal{P}_{m}^+$ are commute with $f$.  
Therefore, the specialization of  $\mathcal{Y}_m^+$ at $q = \xi$ is isomorphic to $\mathcal{P}_{m}^+$ as an $\overline{\mathcal U}_{\xi}(sl_2)$ module.  
\qed
\end{proof}
For an integer $m$ ($1 \leq m \leq N-1$),  we introduce a $4N$ dimensional representation ${\mathcal Y}_m^-$ which is isomorphic to  $W_{2N+m} \oplus W_{2N - m}$.  
The basis of ${\mathcal Y}_m^-$ is $\balpha_0^-$, $\balpha_1^-$, $\cdots$, $\balpha_{2N+m-1}^-$, $\bbeta_0^-$, $\bbeta_1^-$, $\cdots$, $\bbeta_{2N-m-1}^-$, and the actions of $E$, $F$, $K\in {\mathcal U}_q(sl_2)$ are given by
$$
\begin{aligned}
E \, \balpha_i^- &= 
\begin{cases}
[i]_q \, \balpha_{i-1}^- & \text{if $i \leq m$ or $i \geq 2N+1$} ,
\\
[i]_q \, \balpha_{i-1}^- + \left[\begin{matrix}
2N+m-1-i \\[-5pt] 2N-i\end{matrix}\right]_q \, \bbeta_{i-m-1}^- & 
\text{if $m+1 \leq i \leq 2N$},
\end{cases}
\\
E\, \bbeta_i^- &= [i]_q \, \bbeta_{i-1},
\\
F \, \balpha_i^- &=\begin{cases}
[2N+m-1-i]_q \, \balpha_{i+1}^- & \text{if $i \neq m-1$,}
\\
[2N]_q \, \balpha_{i+1}^- +  \left[\begin{matrix}
2N-1 \\[-5pt] 2N\!-\!m\!-\!1\end{matrix}\right]_q  \!\bbeta_{0}^-& 
\text{if $i = m-1$,}
\end{cases}
\end{aligned}
$$
$$
\begin{aligned}
F\, \bbeta_i^- &= [2N\!-\!m\!-\!1\!-\!i]_q \, \bbeta_{i+1}^-, 
\\
K \, \balpha_i^- &= q^{2N+m-1-2i} \, \balpha_i^-, 
\qquad\qquad\qquad
K \, \bbeta_i^- = q^{2N-m-1-2 i} \, \bbeta_i^-.  
\qquad\qquad
\end{aligned}
$$
As for ${\mathcal Y}_{m}^+$, we get the following.  
\begin{proposition}
If $q$ is specialized to $\xi$, then ${\mathcal Y}_m^-$ is isomorphic to the direct sum ${\mathcal P}_{N-m}^- \oplus {\mathcal P}_{N-m}^-$ of the projective module ${\mathcal P}_{N-m}^-$ given by \eqref{eq:pm}.  
\end{proposition}
\begin{proof}
Let  $Y_1$ be the subspace of ${\mathcal Y}_{m}^-$ spanned by $\balpha_0$, $\balpha_1$, $\cdots$, $\balpha_{N+m-1}$, $\bbeta_0$, $\bbeta_1$, $\cdots$, $\bbeta_{N-m-1}$, and let $Y_2$ be the subspace spanned by the remaining basis $\balpha_{N+m}$, $\balpha_{N+m+1}$, $\cdots$, $\balpha_{2N+m-1}$, $\bbeta_{N-m}$, $\bbeta_{N-m+1}$, $\cdots$, $\bbeta_{2N-m-1}$.   
Then $Y_1$ is invariant under the action of $\overline{\mathcal U}_{\xi}(sl_2)$.  
We prove that $Y_1$ and ${\mathcal Y}_{m}^-/Y_1$ are both isomorphic to ${\mathcal P}_{N-m}^-$.  
Let $g$ be a linear map from  ${\mathcal P}_{N-m}^-$ to $Y_1$ defined by 
\begin{equation}
\begin{aligned}
g(\bx_k^-) &=\frac{(-1)^{m+k} [N\!-\!m]!}{[N\!-\!m\!-\!1\!-\!k]!} \,\bbeta_k^-, 
\qquad
g(\by_k^-) = (-1)^{k}\, [k]! \, \balpha_{m+k}^-
\\
&\hspace{6cm}   \text{for $0 \leq k \leq N\!-\!m\!-\!1$},
\\
g(\ba_k^-) &= \frac{[m]}{[m\!-\!1\!-\!k]!} \, \balpha_k^-,
\qquad
g(\bb_k^-) = \frac{(-1)^{N+m+k}\, [N\!-\!1]!}{[m-1-k]!} \, \balpha_{N+k}^-
\\
&\hspace{6cm}
 \text{for $0 \leq k \leq m-1$.}
\end{aligned}
\label{eq:cor2}
\end{equation}
Then, by checking the actions of $K$, $E$, $F$, we see that $g$ gives an isomorphism from $\mathcal{P}_{N-m}^+$ to $Y_1$ as $\overline{\mathcal U}_{\xi}(sl_2)$ modules.  
\par
Next, we define a linear map $h$ from  ${\mathcal P}_{N-m}^-$ to $Y_2$ to show that ${\mathcal P}_{N-m}^-$ are isomorphic to ${\mathcal Y}_{m}^-/Y_1$.
\begin{equation}
\begin{aligned}
h(\bx_k^-) &=\frac{[N-m]!}{ [N-m-1-k]!}\,  \bbeta_{N+k}^-,\qquad
h(\by_k^-) =[k]! \, \balpha_{N+m+k}^-
\\
&\hspace{6cm}
 \text{for $0 \leq k \leq N-m-1$},
\\
h(\ba_k^-) &=\frac{[k]!}{[m\!-\!1]!} \, \bbeta_{N-m+k}^-,
\qquad 
h(\bb_k^-) = \frac{(-1)^k \, [N-1]!}{[m\!-\!1\!-\!k]!} \, \balpha_{2N+k}^-
\\
&\hspace{6cm}
\text{for $0 \leq k \leq m-1$.}
\end{aligned}
\label{eq:cor3}
\end{equation}
Then $h$ defines an isomorphism from $\mathcal{P}_{N-m}^+$ to ${\mathcal Y}_{m}^-/Y_1$.  
This isomorphism induces an inclusion from ${\mathcal P}_{m}^-$ to ${\mathcal Y}_{m}^-$  since ${\mathcal P}_{N-m}^-$ is a projective module.
Hence ${\mathcal Y}_{m}^-$ at $q=\xi$ is isomorphic to $\mathcal{P}_{N-m}^+ \oplus \mathcal{P}_{N-m}^+$.  \qed
\end{proof}
\subsection{Specialization of the $R$ matrix at $q = \xi$}
We can not sepcialize the universal $R$ matrix \eqref{eq:r-matrix} at $q=\xi$ since 
it has a pole at $q=\xi$.  
However, we can sepcialize its action on the representation spaces we are considering.  
\begin{proposition}
\label{prop:R}
Let  $W_m$ $(m=1, 2, \dots)$ and ${\mathcal Y}_m^{\pm}$ $(1 \leq m \leq N-1)$ be the representations of ${\mathcal U}_q(sl_2)$ introduced in Sect. 2.1, and let $V_1$ and $V_2$ be two of these representations.  
If $q$ is generic, 
the universal $R$ matrix of ${\mathcal U}_q(sl_2)$ given by \eqref{eq:r-matrix} acts on $V_1 \otimes V_2$, and this action can be specialized at $q=\xi$.  
The resulting action is the same as the action of the  $R$ matrix of \,$\overline{\mathcal U}_{\xi}(sl_2)$ given by \eqref{eq:r-matrixlog}.   
\end{proposition}
\begin{proof}
For a large $n$, $E^n$ and $F^n$ vanish on $V_1$ and $V_2$.  
Hence it is enough to show that every term of the universal $R$ matrix can be specialized at $q=\xi$.  
We investigate the degrees of zero at $q=\xi$ for the matrix elements of the representations $W_m$  and ${\mathcal Y}_m^{\pm}$.  
By counting the factors of the form $[k N]$ in the matrix elemsts of representations of $E^n$ and $F^n$ constructed in Sect. 2.1, we see that
the zero derees of them at $q=\xi$ are at least $\ell$ if $n \geq \ell N$.  
(Sometimes, they act by 0, whose degree is considered to be $\infty$ for any factor.)
Hence, for $\ell N \leq n < (\ell+1)\, N$,
the zero degree of the term $\frac{1}{\{n\}_q!}\, E^n \otimes F^n$  at $q=\xi$  is a least $\ell$ since its numerator has degree at least $2\ell$  and its denominator has degree $\ell$.  
Hence, the action of  the term $\frac{1}{\{n\}_q!}\, E^n \otimes F^n$  at $q=\xi$ is well-defined even if $n \geq N$, and the terms for $n \geq N$ are all specialized to 0 on $V_1 \otimes V_2$ at $q=\xi$.  \qed
\end{proof}
\subsection{Specialization of the representation on ${\mathcal Y}_m^\pm$ at $q = \xi$}
Let $W_m$ be the irreducible representation of  ${\mathcal U}_q(sl_2)$ introduced  in Sect.  1.3, 
and let $\rho_m$, $\eta_m^\pm$ be the homomorphisms from ${\mathcal U}_q(sl_2)$ to $W_m$ and ${\mathcal Y}_m^\pm$ respectively.  
Moreover, for a knot $L$, we define $\rho_m(L)$ and $\eta_m^\pm(L)$ by
$$
\rho_m(L) = \widetilde\tr\Big(\rho_{W_m}(b_L)\Big), \qquad
\eta_m^\pm(L) = \widetilde\tr\left(\rho_{{\mathcal Y}_m^\pm}(b_L)\right),
$$
where $\rho_{{\mathcal Y}_m^\pm}(b_L)$ is the representation of $b_L$ on ${\mathcal Y}_m^\pm \otimes \cdots \otimes {\mathcal Y}_m^\pm$ given in Sect.  1.5.  
Then $\rho_m(L)$ and $\eta_m^\pm(L)$ are elements of $\End(W_m)$ and $\End({\mathcal Y}_m^+)$ respectively.    
Moreover, $\rho_m(L)$ is a scalar matrix such that the corresponding scalara is the normalized colored Jones invariant $\widetilde V_m(L)$.  
The matrix $\eta_m^\pm(L)$ may not be scalar matrix since ${\mathcal Y}_m^\pm$ is not irreducible.  
However, for any $g \in {\mathcal U}_q(sl_2)$, $\eta_m^\pm(L)$  satifies 
$
\eta_m^\pm(g) \, \eta_m^\pm(L) = \eta_m^\pm(L) \, \eta_m^\pm(g)
$,   
and 
 $\eta_m^\pm(L)$ is given by the following with some scalars $x_k^\pm$.  
$$
\begin{aligned}
\eta_m^+(L) \, \balpha_{k}^+ 
&=
\begin{cases}
\widetilde V_{2N-m}(L) \, \balpha_{k}^+, & (0 \leq k \leq N-m-1)
\\
\widetilde V_{2N-m}(L) \, \balpha_{k}^+ + x_{k-N+m}^+ \, \bbeta_{k-N+m}^+,
& (N-m \leq k \leq N-1)
\\
\widetilde V_{2N-m}(L) \, \balpha_{k}^+, & (N \leq k \leq 2N-m-1)
\end{cases}
\\
\eta_m^+(L) \, \bbeta_{k}^+&
= \widetilde V_m(L) \, \bbeta_{k}^+, \qquad\qquad(0 \leq k \leq m-1)
\end{aligned}
$$
$$
\begin{aligned}
\eta_m^-(L) \, \balpha_{k}^- 
&=
\begin{cases}
\widetilde V_{2N+m}(L) \, \balpha_{k}^-, & (0 \leq k \leq m-1)
\\
\widetilde V_{2N+m}(L) \, \balpha_{k}^- + x_{k-m}^+ \, \bbeta_{k-m}^+,
\qquad
& (N-m \leq k \leq N-1)
\\
\widetilde V_{2N+m}(L) \, \balpha_{k}^-, & (N \leq k \leq 2N+m-1)
\end{cases}
\\
\eta_m^-(L) \, \bbeta_{k}^-&
= \widetilde V_{2N-m}(L) \, \bbeta_{k}^-. \qquad\qquad(0 \leq k \leq 2N-m-1)
\end{aligned}
$$
\par
Now, we obtain $x_0^\pm$.  
We have
\begin{multline*}
\eta_m^+(F) \, \eta_m^+(L) \, \balpha_{N-m-1}^+
=
\widetilde V_{2N-m}(L) \, \eta_m^+(F) \, \balpha_{N-m-1}^+
=
\\
\widetilde V_{2N-m}(L) \, 
\left([N]_q\, \balpha_{N-m}^+ + 
\left[\begin{matrix}N-1 \\ m-1\end{matrix}\right]_q \, \bbeta_0^+\right), 
\end{multline*}
and
\begin{multline*}
\eta_m^+(L) \,\eta_m^+(F)\,  \balpha_{N-m-1} ^+
=
\eta_m^+(L) \,\left([N]_q\, \balpha_{N-m} ^+ 
+
\left[\begin{matrix}N\!-\!1 \\ m\!-\!1\end{matrix}\right]_q \, \bbeta_0^+\right)
=
\\
[N]_q\, \widetilde V_{2N-m}(L) \, \balpha_{N-m}^+
+
[N]_q x_0^+  \bbeta_0^+
+ 
\widetilde V_m(L)\,
\left[\begin{matrix}N\!-\!1 \\ m\!-\!1\end{matrix}\right]_q  \bbeta_0^+.
\end{multline*}
Hence we get
\begin{equation}
x^+ = 
\left[\begin{matrix}N-1 \\ m-1\end{matrix}\right]_q \, \dfrac{\widetilde V_{2N-m}(L) - 
\widetilde V_{m}(L)}{[N]_q}. 
\label{eq:xp}
\end{equation}
We also get $x_0^-$ similarly as follows.  
\begin{multline*}
\eta_m^-(F) \, \eta_m^-(L) \, \balpha_{m-1}^-
=
\widetilde V_{2N+m}(L) \, \eta_m^-(F) \, \balpha_{m-1}^-
=
\\
\widetilde V_{2N+m}(L)\left([2N]_q\, \balpha_{m}^- + 
\left[\begin{matrix}2N-1 \\ 2N-m-1\end{matrix}\right]_q \, \bbeta_0^-\right), 
\end{multline*}
which is equal to  
\begin{multline*}
\eta_m^-(L)\, \eta_m^-(F) \,    \balpha_{m-1}^- 
=
\eta_m^-(L)\, \left(
[2N]_q \, \balpha_{m}^- +  \left[\begin{matrix}
2N-1 \\ 2N\!-\!m\!-\!1\end{matrix}\right]_q  \!\bbeta_{0}^-
\right)
=
\\
\widetilde V_{2N+m}(L) \, [2N]_q\, \balpha_{m}^-
+
[2N]_qx_0^- \, \bbeta_0^-
+ 
\widetilde V_{2N-m}(L)
\left[\begin{matrix}2N-1 \\ 2N\!-\!m\!-\!1\end{matrix}\right]_q \, \bbeta_0^-.
\end{multline*}
Hence we get
\begin{equation}
x^- = 
\left[\begin{matrix}2N-1 \\ m\end{matrix}\right]_q \, \dfrac{\widetilde V_{2N+m}(L) - 
\widetilde V_{2N-m}(L)}{[2N]_q}. 
\label{eq:xm}
\end{equation}
By specializing $q$ to  $\xi$ at \eqref{eq:xp} and \eqref{eq:xm},   
 we get
$$
\begin{aligned}
\lim_{q\to\xi} x_0^+ 
&=
\left.
-\dfrac{\xi}{2 \, N} \, \frac{d}{dq}  \{1\}_q\big( \widetilde V_{2N-m}(L) - 
 \widetilde V_{m}(L)\big)\right|_{q=\xi}, 
\\
\lim_{q\to\xi} x_0^- 
&=
\dfrac{(-1)^{m}\, \xi}{4\, N} \, 
\left.\frac{d}{dq} \{1\}_q\big(\widetilde  V_{2N+m}(L) - 
\widetilde V_{2N-m}(L)\big)\right|_{q=\xi},
\end{aligned}
$$
by using l'Hopital's rule.
Now we are ready to prove Propsition \ref{prop:b}.  
\subsection{Proof of Proposition \ref{prop:b}} 
\label{ss:proof}
We compare $x_0^\pm$ with  the coefficients $b_m^{\pm}(L)$ of $\bw_s^\pm$ introduced in \cite{MN}.  
Let $\widetilde\eta_{m}^+$,  $\widetilde\eta_{m}^-$ be the representations on ${\mathcal P}_{m}^+$ and ${\mathcal P}_{N-m}^-$ in \cite{MN}, and 
$\widetilde\eta_{m}^+(L)$, $\widetilde\eta_{m}^-(L)$ be the elements of $\End({\mathcal P}_{m}^+)$ and $\End({\mathcal P}_{N-m}^-)$ defined as $\eta_m^\pm(L)$.  
Then we have
$$
\begin{aligned}
\widetilde\eta_{m}^+(L) \, \bb_0^+
&= 
\left.\widetilde V_{2N-m}(L)\right|_{q=\xi}\,  \bb_0^+ + b_{m}^+(L) \, \ba_0^+
, 
\\
\widetilde\eta_{m}^-(L) \, \by_0^-
&= 
\left.\widetilde V_{2N+m}(L)\right|_{q=\xi} \,  \by_0^- + b_{m}^-(L) \, \bx_0^-.
\end{aligned}
$$   
From Proposition \ref{prop:R}, $\widetilde\eta_m^\pm(L)$ are essentially the same as the specialization of $\eta_m^\pm(L)$ at $q =\xi$.  
Since an isomorphism between two isomorphic projective modules is uniquely determined up to a scalar multiple, the matrices of $\widetilde\eta_m^\pm(L)$ and $\eta_m^\pm(L)$ shold be related by the isomorphisms $f$, $g$ in  \eqref{eq:cor1}, \eqref{eq:cor2} as follows.  
$$
\widetilde\eta_m^+(L) = f^{-1}\circ \left.\eta_m^+(L)\right|_{q=\xi}\circ f, 
\qquad
\widetilde\eta_m^-(L) = g^{-1}\circ \left.\eta_m^-(L)\right|_{q=\xi}\circ g. 
$$
Hence, by using
$$f(\bb_0^+) = \balpha_{N-m}^+,\ \  
f(\ba_0^+) = [m]\, \bbeta_0^+,\ \  
g(\by_0^-) = \balpha_{m}^-, \ \ 
g(\bx_0^-) = (-1)^{m}\, [m] \, \bbeta_0^-, 
$$
we have  
\begin{equation}
\begin{aligned}
b_{m}^+ (L)
&=
\left.
\dfrac{\xi}{2 \, N\, [m]} \,\frac{d}{dq}  \{1\}_q\big(\widetilde V_{m}(L) - 
\widetilde V_{2N-m}(L)\big)\right|_{q=\xi},
\\
b_{m}^-(L) 
&=
\left.
\dfrac{\xi}{4 \, N\,  [m]} \,\frac{d}{dq} \{1\}_q\big( \widetilde V_{2N+m}(L) - 
\widetilde V_{2N-m}(L)\big)\right|_{q=\xi}. \qquad\qquad\qed
\end{aligned}
\label{eq:centervalue}
\end{equation}
\subsection{Habiro's formula}
Let $s$ be an integer satisfying $1 \leq s \leq N-1$ and
put 
$$
\underline s = \min(s, N-s),\qquad 
\overline s = \max(s, N-s).
$$
By using Habiro's universal formula \eqref{eq:universal},  $b_{s}^\pm(L)$ is expressed in terms of $a_i(L)$ as follows.    
We put $a_i(L)_\xi = \left.a_i(L)\right|_{q=\xi}$.  
\begin{proposition}
\label{prop:Habiro}
For a knot $L$, we have
\begin{multline}
b_{s}^+(L)= b_{s}^-(L) =
\\
\dfrac{\{1\}^2}{\{s\}}\! \left(\!
\sum_{i=0}^{\underline s -1} 
\frac{a_i(L)_\xi  \{s+i\}!}{[s]\{s-i-1\}!}\!
\sum_{{\genfrac{}{}{0pt}{}{s-i \leq k \leq s+i}
{k \neq s}}}\!\!\!\!
\frac{\{k\}_+}{\{k\}} 
 +2\sum_{i={\underline s}}^{\overline s -1}
a_i(L)_\xi \widetilde{\{s+i, i\}}\widetilde{\{s-1, i\}}\!\right)\!,
\end{multline}
where $\widetilde{\{n, j\}}$ is given by the following.
\begin{equation}
\widetilde{ \{n, j\}}  = 
\prod_{\genfrac{}{}{0pt}{}{0\leq k \leq j-1}{n-k = Nt}}(-1)^t
\prod_{\genfrac{}{}{0pt}{}{0\leq k \leq j-1}{n-k \notin N{\mathbf Z}}}\{n-k\}_q.  
\label{eq:tilde}
\end{equation}
\end{proposition}
\begin{corollary}
\label{cor:cor}
Let $L^f$ be the framed knot with framing $f$ which is isotopic to $L$ as a non-framed knot.  
The colored Jones invariant $V_m$ is generalized to a framed knot by
$V_m(L^f) = q^{\frac{m^2-1}{2}f} \, V_m(L)$, 
and the invariants $b_{s}^+(L^f)$ and $b_{s}^-(L^f)$ 
are generalized as follows.  
\begin{equation}
\begin{aligned}
b_{s}^+(L^f)
&=
q^{\frac{s^2-1}{2}f} \, b_{s}^+(L) 
+ 
\dfrac{(-N + s) \, f \, 
\{1\} }{[s]^2}\, \left.V_{s}(L)\right|_{q=\xi},\\
b_{s}^-(L^f) 
&=
q^{\frac{s^2-1}{2}f} \, b_{s}^-(L) 
+ 
\dfrac{s\, f \, 
\{1\} }{[s]^2}\, \left.V_{s}(L)\right|_{q=\xi}.  
\end{aligned}
\end{equation}
\end{corollary}
\par\noindent{\it Proof of Proposition \ref{prop:Habiro}.}
Let $\tilde a_i^{}(L) = \{1\}_q \, a_i(L)$ for $a_i(L)$ in  \eqref{eq:universal}
 and $\tilde a_i^{}(L)_\xi = \left.\tilde a_i^{}(L)\right|_{q=\xi}$,  
then
$$
\dfrac{d}{dq} \left. \dfrac{\{1\}_q V_{s}(L)}{[s]_q}\right|_{q = \xi}
= 
\sum_{i=0}^{s-1} \frac{\tilde a_i^{}(L)_\xi \, \{s+i\}!}{\{s\} \,\{ s-i-1\}!}
\left( \left.\dfrac{\frac{d}{dq} \, \tilde a_i^{}(L)}{\tilde a_i^{}(L)}\right|_{q=\xi} + 
\sum_{{\genfrac{}{}{0pt}{}{s-i \leq k \leq s+i}
{k \neq s}}}
\frac{k\,\{k\}_+}{\xi \, \{k\}}\right), 
$$
where $\{k\}_+ = \xi + \xi^{-1}$.  
Now we compute  \eqref{eq:differential} 
by usinge
$\{2N-k\} = -\{k\}$ and 
$
\left.\frac{d}{dq}F(q)\right|_{q=\xi} 
= 
- \frac{2\, N}{\xi} \, \lim_{q\to\xi}
\frac{F(q)}{\{N\} }  
$ for a function $F(q)$ of $q$.  
$$
\begin{aligned}
\frac{d}{dq} & \{1\}_q \left.\left(\frac{V_{s}(L)}{[s]_q}\,  - 
\frac{V_{2N-s}(L)}{[2N-s]_q} \, \right) \right|_{q=\xi} 
= 
\qquad\qquad\qquad\qquad
\qquad\qquad\qquad\qquad
\\
&\sum_{i=0}^{\underline s-1} 
\frac{\tilde a_i^{}(L)_\xi \, \{s+i\}!}{\{s\} \,\{ s-i-1\}!}
 \left( \left.\dfrac{\frac{d}{dq} \, \tilde a_i^{}(L)}{\tilde a_i^{}(L)}\right|_{q=\xi} + 
\sum_{{\genfrac{}{}{0pt}{}{s-i \leq k \leq s+i}
{k \neq s}}}
\frac{ k\,\{k\}_+}{\xi \, \{k\}}\right) 
\\&\qquad\qquad\qquad
-
\sum_{i={\underline s}}^{s-1}
\frac{2 \, N}{\xi} \lim_{q\to\xi} \dfrac{\tilde a_i^{}(L) \,  \{s+i, i\}_q \, \{s-1, i\}_q}{\{N\}_q}
\end{aligned}
$$
$$
\begin{aligned}
-&
\sum_{i=0}^{\underline s-1} 
\frac{\tilde a_i^{}(L)_\xi \, \{2N-s+i\} !}{\{2N-s\}\,\{2N-s-i-1\}!}
\left(\left. \dfrac{\frac{d}{dq} \, 
\tilde a_i^{}(L)}{\tilde a_i^{}(L)}\right|_{q = \xi} + 
\sum_{\genfrac{}{}{0pt}{}{2N-s-i\leq k \leq 2N-s+i}{k \neq 2N-s}} 
\dfrac{k \, \{k\}_+}{\xi \, \{k\}}  \right) 
\\&\qquad\qquad\qquad
+
\sum_{i={\underline s}}^{\overline s-1} 
 \frac{2 \, N}{\xi}\lim_{q\to\xi} \dfrac{\tilde a_i^{}(L) \, \{2N-s+i, i\}_q \, \{2N-s-1, i\}_q}{\{N\}_q}.  
 \end{aligned}
 $$
For $\underline s \leq i \leq s-1$, we have
\begin{equation}
\lim_{q\to\xi} \frac{ \{s+i, i\}_q \, \{s-1, i\}_q}{\{N\}_q}
=
\begin{cases}
0 & \text{if $s \leq \frac{N}{2}$}, \\
-\widetilde{\{s+i, i\}}\, \{s-1, i\} & \text{if $s > \frac{N}{2}$},
\end{cases}
\label{eq:relation1}
\end{equation}
and for $\underline s \leq i \leq \overline s-1$, 
\begin{equation}
\lim_{q\to\xi} \dfrac{\{2N-s+i, i\}_q \, \{2N-s-1, i\}_q}{\{N\}_q}
=
\begin{cases}
2  \{s+i, i\} \widetilde{\{s-1, i\}} & \text{if $s \leq \frac{N}{2}$}, \\
\widetilde{\{s+i, i\}}\, \{s-1, i\} & \text{if $s > \frac{N}{2}$}.
\end{cases}
\label{eq:relation2}
\end{equation}
We also know that
\begin{equation}
 \{2N-k\} = -\{k\}, 
 \qquad
\frac{\{2N-s+i\}!}{\{2N-s\} \,\{2N-s-i-1\}!}
=
\frac{\{s+i\}!}{\{s\} \,\{ s-i-1\}!}.
\label{eq:relation3}
\end{equation}
By using \eqref{eq:relation1}, \eqref{eq:relation2}, \eqref{eq:relation3}, we get the following.  
If $s \leq N/2$, then
$$
\begin{aligned}
\frac{d}{dq} & \{1\}_q \left.
\left(\frac{V_{s}(L)}{[s]_q}\,  - 
\frac{V_{2N-s}(L)}{[2N-s]_q} \, \right) \right|_{q=\xi} 
\\
=&\sum_{i=0}^{\underline s-1} 
\frac{\tilde a_i^{}(L)_\xi \, \{s+i\}!}{\{s\} \,\{ s-i-1\}!}
 \left( \left.\dfrac{\frac{d}{dq} \, \tilde a_i^{}(L)}{\tilde a_i^{}(L)}\right|_{q=\xi} + 
\sum_{{\genfrac{}{}{0pt}{}{s-i \leq k \leq s+i}
{k \neq s}}}
\frac{ k\,\{k\}_+}{\xi \, \{k\}}\right) 
\\
&-
\sum_{i=0}^{\underline s-1} 
\frac{\tilde a_i^{}(L)_\xi \, \{s+i\} !}{\{s\}\,\{s-i-1\}!}
\left(\left. \dfrac{\frac{d}{dq} \, 
\tilde a_i^{}(L)}{\tilde a_i^{}(L)}\right|_{q = \xi} - 
\sum_{\genfrac{}{}{0pt}{}{s-i\leq k \leq s+i}{k \neq s}} 
\dfrac{(2N-k) \, \{k\}_+}{\xi \, \{k\}}  \right) 
\\&\qquad\qquad\qquad
+
\frac{4 \, N}{\xi}
\sum_{i={\underline s}}^{\overline s-1} 
 \tilde a_i^{}(L)_\xi \, {\{s+i, i\}} \, \widetilde{\{s-1, i\}}
\\
=&\dfrac{2\, N}{\xi}
\sum_{i=0}^{\underline s-1} 
\frac{\tilde a_i^{}(L)_\xi  \{s+i\}!}{\{s\} \,\{ s-i-1\}!}
\sum_{{\genfrac{}{}{0pt}{}{s-i \leq k \leq s+i}
{k \neq s}}}
 \dfrac{\{k\}_+}{\{k\}} 
+
\dfrac{4\, N}{\xi} 
\sum_{i={\underline s}}^{\overline s-1}
\tilde a_i^{}(L)_\xi  \widetilde{\{s+i, i\}} \widetilde{\{s-1, i\}}. \end{aligned}
 $$
 If $s > N/2$, then
 $$
\begin{aligned}
\frac{d}{dq} & \{1\}_q \left.
\left(\frac{V_{s}(L)}{[s]_q}\,  - 
\frac{V_{2N-s}(L)}{[2N-s]_q} \, \right) \right|_{q=\xi} 
\\
=&\sum_{i=0}^{\underline s-1} 
\frac{\tilde a_i^{}(L)_\xi \, \{s+i\}!}{\{s\} \,\{ s-i-1\}!}
 \left( \left.\dfrac{\frac{d}{dq} \, \tilde a_i^{}(L)}{\tilde a_i^{}(L)}\right|_{q=\xi} + 
\sum_{{\genfrac{}{}{0pt}{}{s-i \leq k \leq s+i}
{k \neq s}}}
\frac{ k\,\{k\}_+}{\xi \, \{k\}}\right) 
\\&\qquad\qquad\qquad
+
\frac{2 \, N}{\xi} \sum_{i={\underline s}}^{s-1}
 \tilde a_i^{}(L)_\xi \,  \widetilde{\{s+i, i\}}\, \{s-1, i\}
\\
&-
\sum_{i=0}^{\underline s-1} 
\frac{\tilde a_i^{}(L)_\xi \, \{s+i\} !}{\{s\}\,\{s-i-1\}!}
\left(\left. \dfrac{\frac{d}{dq} \, 
\tilde a_i^{}(L)}{\tilde a_i^{}(L)}\right|_{q = \xi} - 
\sum_{\genfrac{}{}{0pt}{}{s-i\leq k \leq s+i}{k \neq s}} 
\dfrac{(2N-k) \, \{k\}_+}{\xi \, \{k\}}  \right) 
\\&\qquad\qquad\qquad 
+
\frac{2 \, N}{\xi}
\sum_{i={\underline s}}^{\overline s-1} 
 \tilde a_i^{}(L)_\xi \, \widetilde{{\{s+i, i\}}} \, {\{s-1, i\}}
 \end{aligned}
 $$
 $$
 \begin{aligned}
=&\dfrac{2\, N}{\xi}
\sum_{i=0}^{\underline s-1} 
\frac{\tilde a_i^{}(L)_\xi  \{s+i\}!}{\{s\} \,\{ s-i-1\}!}
\sum_{{\genfrac{}{}{0pt}{}{s-i \leq k \leq s+i}
{k \neq s}}}
 \dfrac{\{k\}_+}{\{k\}} 
+
\dfrac{4\, N}{\xi} 
\sum_{i={\underline s}}^{\overline s-1}
\tilde a_i^{}(L)_\xi  \widetilde{\{s+i, i\}} \widetilde{\{s-1, i\}}. 
\end{aligned}
 $$
 Therefore, for all $s$ with $1 \leq s \leq N-1$, we have
\begin{multline}
 \frac{d}{dq} \, \{1\}_q \left.
\left(\frac{V_{s}(L)}{[s]_q}\,  - 
\frac{V_{2N-s}(L)}{[2N-s]_q} \, \right) \right|_{q=\xi} 
=
\\
\dfrac{2\, N}{\xi}
\sum_{i=0}^{\underline s-1} 
\frac{\tilde a_i^{}(L)_\xi  \{s+i\}!}{\{s\} \,\{ s-i-1\}!}
\!\!\!\!
\sum_{{\genfrac{}{}{0pt}{}{s-i \leq k \leq s+i}
{k \neq s}}}
\!\!\!\!\!
 \dfrac{\{k\}_+}{\{k\}} 
+
\dfrac{4\, N}{\xi} 
\sum_{i={\underline s}}^{\overline s-1}
\tilde a_i^{}(L)_\xi  \widetilde{\{s+i, i\}} \widetilde{\{s-1, i\}}. 
\label{eq:bp}
\end{multline}
Similarly we get
\begin{multline}
\left.
\frac{d}{dq} \{1\}_q \left(\frac{V_{2N+s}(K)}{[2N+s]_q} - 
\frac{V_{2N-s}(L)}{[2N-s]_q}\right) \right|_{q=\xi} 
=
\\
\dfrac{4\, N}{\xi} \sum_{i=0}^{\underline s-1} 
\frac{\tilde a_i^{}(L)_\xi \{s+i\}!}{\{s\} \{ s-i-1\}!}
\!\!\!\!
\sum_{{\genfrac{}{}{0pt}{}{s-i \leq k \leq s+i}
{k \neq s}}}
\!\!\!\!\!
 \dfrac{\{k\}_+}{\{k\}} 
+
\dfrac{8\, N}{\xi}
\sum_{i={\underline s}}^{\overline s-1}\!
\tilde a_i^{}(L)_\xi  \widetilde{\{s+i, i\}}  \widetilde{\{s-1, i\}},
\label{eq:bm}
\end{multline}
by using \eqref{eq:relation2}, \eqref{eq:relation3},
\begin{equation}
\lim_{q\to\xi} \dfrac{\{2N+s+i, i\}_q \, \{2N+s-1, i\}_q}{\{N\}_q}
=
\begin{cases}
2  \{s+i, i\}\widetilde{\{s-1, i\}} & \text{if $s \leq \frac{N}{2}$}, \\
3  \widetilde{\{s+i, i\}}\{s-1, i\} & \text{if $s > \frac{N}{2}$},
\end{cases}
\label{eq:relation4}
\end{equation}
and
\begin{equation}
\frac{\{2N+s+i\}!}{\{2N+s\} \,\{2N+s-i-1\}!}
=
\frac{\{s+i\}!}{\{s\} \,\{ s-i-1\}!}.
\label{eq:relation5}
\end{equation}
Combining \eqref{eq:centervalue} with \eqref{eq:bp} and \eqref{eq:bm}, 
we get
$$
\begin{aligned}
&b_{s}^+(L) = b_{s}^-(L) =
\\
&
\dfrac{\{1\}^2}{\{s\}} \! \left(
\sum_{i=0}^{\underline s-1} 
\frac{a_i(L)_\xi \{s+i\}!}{\{s\} \{ s-1-i\}!}\!\!
\sum_{{\genfrac{}{}{0pt}{}{s-i \leq k \leq s+i}
{k \neq s}}}\!\!
 \dfrac{ \{k\}_+}{\{k\}} 
\right.
\left.\!\!
 +2\sum_{i={\underline s}}^{\overline s-1}
a_i(L)_\xi  \,\widetilde{\{s+i, i\}} \, \widetilde{\{s-1, i\}}\right).\ 
\qed
\end{aligned}
$$
\par\noindent{\it Proof of Corollary \ref{cor:cor}.}
Since $V_{s}(L^f) = q^{\frac{s^2-1}{2}f}\,V_{s}(L)$, 
$\left.V_s(L^f)\right|_{q=\xi} = \left.V_{2N-s}(L^f)\right|_{q=\xi}$
$=
\left.V_{2N+s}(L^f)\right|_{q=\xi}$, 
and $b_s^+(L^f)$, $b_s^-(L^f)$ are given by \eqref{eq:centervalue}, we have
$$
\begin{aligned}
b_{s}^+(L^f) 
&=
\xi^{\frac{s^2-1}{2}f} \, b_{s}^+(L) 
+ 
\dfrac{1}{2N} 
\left(\frac{s^2-1}{2} - \frac{(2N-s)^2 -1}{2}\right) 
\frac{f \, 
\{1\}}{[s]^2}  \left.V_{s}(L)\right|_{q=\xi}
\\
&=
\xi^{\frac{s^2-1}{2}f} \, b_{s}^+(L) 
- 
f\,  
\frac{( N-s) \, \{1\} }{[s]^2} \left.V_{s}(L)\right|_{q=\xi},
\\
b_{s}^-(L^f) 
&= 
\xi^{\frac{s^2-1}{2}f}\, b_{s}^-(L) 
+ 
f \, 
\frac{s \, \{1\} }{[s]^2}\left.V_{s}(L)\right|_{q=\xi}.
\hspace{5cm} \qed
\end{aligned}
$$
\subsection{Coefficients of centers}
The center of the restricted quantum group $\overline{\mathcal U}_{\xi}(sl_2)$ are spanned by $\be_0$, $\cdots$, $\be_{N}$, $\bw_1^{\pm}$, $\cdots$, $\bw_{N-1}^{\pm}$, where $\be_s$ is the central idempotent and $\bw_s^\pm$ is the center in the radical of ${\mathcal P}_s^\pm$. 
For a framed knot $L^f$, let $z = z({L^f})$ be the center of $\overline{\mathcal U}_\xi(sl_2)$  determined from $L^f$ by using a tangle $T_{L^f}$ obtained from $L^f$.  
Then $z$ is expressed as a linear combination of the good basis \eqref{eq:firstbasis}. 
 $$
z({L^f}) =
\sum_{s=1}^{N-1}
\alpha_s^{(N)}({L^f}) \, \hat \brho_s 
+
\sum_{s=1}^{N-1}
\beta_s^{(N)}({L^f}) \, \hat \bvarphi_s 
+
\sum_{s=0}^N
\gamma_s^{(N)}({L^f}) \, \hat\bkappa_s.
$$  
In the following, $L$ represents the non-framed knot which is historic to $L^f$.  
By using \eqref{eq:goodbasis}, we have
$$
\begin{aligned}
&\alpha_s^{(N)}(L^f) 
= 
(-1)^{N+s} \, N \, \{1\} \, \left.V_{s}(L^f)\right|_{q=\xi}, 
\qquad\qquad\qquad\quad\qquad\ 
1 \leq s \leq N-1,
\\
&\beta_s^{(N)}(L^f)
= 
[s]^2 \, \left(b_{s}^+(L^f) - b_{s}^-(L^f)\right)
=
-N\, f \, \{1\} \, \left.V_{s}(L^f)\right|_{q=\xi},  
\ 
1 \leq s \leq N-1,
\\
&\gamma_0^{(N)}(L^f)
=
\left.\dfrac{V_{2N}(L^f)}{[2N]}\right|_{q=\xi}
=
\frac{\xi^{1-\frac{f}{2}}}{4\,N} 
\left.\frac{d}{dq} \{1\}_q V_{2N}(L)\right|_{q=\xi}
\\&
\qquad\qquad\qquad\qquad\qquad\qquad\qquad\qquad\quad
=
\left.
\dfrac{N \, \xi^{-\frac{f}{2}}\,\{1\}}{2 \, \pi\,\sqrt{-1}}\,
\dfrac{d}{dm}
V_m(L)\right|_{\genfrac{}{}{0pt}{}{m=2N}{q=\xi}},
\\
&\gamma_s^{(N)}(L^f) 
= 
[s]^2 \, \left(\dfrac{s}{N}\, b_{s}^+(L^f) 
+ 
\dfrac{N-s}{N}\, b_{s}^-(L^f)\right)
+
\dfrac{\{s\}_+}{[s]} \left.V_{s}(L^f)\right|_{q=\xi}
\\
&=
\xi^{\frac{s^2-1}{2}f} \left(
\sum_{i=0}^{\underline s-1} 
a_i(L)_\xi  \{s+i, 2i+1\} 
\sum_{k=s-i}^{s+i} \dfrac{\{k\}_+}{\{k\}} 
+
2 \sum_{i={\underline s}}^{\overline s-1}
a_i(L)_\xi  \widetilde{\{s+i, 2i+1\}} 
\right)
\\
&=
\frac{\xi}{2N} \,\xi^{\frac{s^2-1}{2}f}
\left.\frac{d}{dq} \{1\}_q\big(V_{s}(L) + V_{2N-s}(L)\big)\right|_{q=\xi}
\\
&=
\left.
\dfrac{N \, \{1\}}{\pi\,\sqrt{-1}}\,
\xi^{\frac{s^2-1}{2}f} \, 
\dfrac{d}{dm}
V_m(L)\right|_{\genfrac{}{}{0pt}{}{m=s}{q=\xi}},
\qquad\qquad\qquad\qquad\qquad\quad
1 \leq s \leq N-1,
\\
&\gamma_N^{(N)}(L^f)
=
-\left.\dfrac{V_{N}(L^f)}{[N]}\right|_{q=\xi}
=
\frac{\xi^{1+\frac{N^2-1}{2}f}}{2 \, N} 
\left.\frac{d}{dq} \{1\}_q V_{N}(L)\right|_{q=\xi}
\\&
\qquad\qquad\qquad\qquad\qquad\qquad\qquad\qquad\quad
=
\left.
\dfrac{N \, \xi^{\frac{N^2-1}{2}f}\,\{1\}}{2 \, \pi\,\sqrt{-1}}\,
\dfrac{d}{dm}
V_m(L)\right|_{\genfrac{}{}{0pt}{}{m=N}{q=\xi}}
.
\end{aligned}
$$
In  $\frac{d}{dm}\, V_m(L)$, the colored Jones invariant $V_m(L)$ is expressed by Habiro's universal formula \eqref{eq:universal} and considered to be an infinite sum with the variable $m$. 
The integer $s$ is substituted after obtaining the derivative.  
The sum reduces to a finite sum when $q$ is specialized to $\xi$.    
Hence we get the following. 
\begin{theorem}
For framed knot $L^f$ with framing $f$ and let $L$ be the same knot without framing.  
Then we have
$$
\begin{aligned}
\alpha_s^{(N)}(L^f) &= (-1)^{N+s} \, N \, \{1\} \, \left.V_{s}(L^f)\right|_{q=\xi}, 
\qquad\qquad\qquad\ 
1 \leq s \leq N-1,
\end{aligned}
$$
$$
\begin{aligned}
\beta_s^{(N)}(L^f) &= 
-N\, f \, \{1\} \, \left.V_{s}(L^f)\right|_{q=\xi},
\qquad\qquad\qquad\qquad\qquad\qquad\qquad\quad\ \ 
1 \leq s \leq N-1,
\\
\gamma_0^{(N)}(L^f)
&=
\left.\dfrac{V_{2N}(L^f)}{[2N]}\right|_{q=\xi}
=
\frac{
\xi^{1-\frac{f}{2}}}{4\,N} 
\left.\frac{d}{dq} \{1\}_q V_{2N}(L)\right|_{q=\xi}
=
\left.
\dfrac{N \, \xi^{-\frac{f}{2}}\,\{1\}}{2 \, \pi\,\sqrt{-1}}\,
\dfrac{d}{dm}
V_m(L)\right|_{\genfrac{}{}{0pt}{}{m=2N}{q=\xi}},
\\
\gamma_s^{(N)}(L^f)  
&=
\frac{\xi^{1+\frac{s^2-1}{2}f}}{2N} \,
\left.\frac{d}{dq} \{1\}\big(V_{s}(L) + V_{2N-s}(L)\big)\right|_{q=\xi}
=
\left.
\dfrac{N \, \{1\}}{\pi\,\sqrt{-1}}\,
\xi^{\frac{s^2-1}{2}f} \, 
\dfrac{d}{dm}
V_m(L)\right|_{\genfrac{}{}{0pt}{}{m=s}{q=\xi}},
\\&
\qquad\qquad\qquad\quad\qquad\qquad\qquad
\qquad\qquad\qquad\qquad\qquad\qquad\quad\ \ 
1 \leq s \leq N-1,
\\
\gamma_N^{(N)}(L^f)
&=
-\left.\dfrac{V_{N}(L^f)}{[N]}\right|_{q=\xi}\!\!\!
=
\frac{\xi^{1+\frac{N^2-1}{2}f}}{2 \, N} 
\left.\frac{d}{dq} \{1\}_q V_{N}(L)\right|_{q=\xi}\!\!\!
=
\left.
\dfrac{N \, \xi^{\frac{N^2-1}{2}f}\,\{1\}}{2 \, \pi\,\sqrt{-1}}\,
\dfrac{d}{dm}
V_m(L)\right|_{\genfrac{}{}{0pt}{}{m=N}{q=\xi}}
.
\end{aligned}
$$
Especially, if the framing $f=0$, 
$$
\begin{aligned}
\alpha_s^{(N)}(L) &= (-1)^{N+s} \, N \, \{1\} \, \left.V_{s}(L)\right|_{q=\xi}, 
\qquad\qquad\qquad\qquad\qquad\qquad\qquad\ 
1 \leq s \leq N-1,
\\
\beta_s^{(N)}(L) &= 0,
\qquad\qquad\qquad\qquad\qquad\qquad\qquad 
\qquad\qquad \ \ \qquad\qquad\qquad
1 \leq s \leq N-1,
\\
\gamma_0^{(N)}(L)
&=
\left.\dfrac{V_{2N}(L)}{[2N]}\right|_{q=\xi}\!\!
=
\frac{\xi}{4\,N} 
\left.\frac{d}{dq} \{1\}_q V_{2N}(L)\right|_{q=\xi}\!\!
=
\left.
\dfrac{N \,\{1\}}{2 \, \pi\,\sqrt{-1}}\,
\dfrac{d}{dm}
V_m(L)\right|_{\genfrac{}{}{0pt}{}{m=2N}{q=\xi}},
\\
\gamma_s^{(N)}(L) 
&=
\frac{\xi}{2N} 
\left.\frac{d}{dq} 
\{1\}\big(V_{s}(L) + V_{2N-s}(L)\big)
\right|_{q=\xi}
=
\left.
\dfrac{N \, \{1\}}{\pi\,\sqrt{-1}}\,
\dfrac{d}{dm}
V_m(L)\right|_{\genfrac{}{}{0pt}{}{m=s}{q=\xi}},  
\\&\qquad\qquad\qquad\qquad\qquad\qquad\qquad
\qquad\qquad\qquad \qquad\quad\qquad\qquad
1 \leq s \leq N-1,
\\
\gamma_N^{(N)}(L)
&=
-\left.\dfrac{V_{N}(L)}{[N]}\right|_{q=\xi}\!\!
=
\frac{\xi}{2 \, N} 
\left.\frac{d}{dq} \{1\}_q V_{N}(L)\right|_{q=\xi}\!\!
=
\left.
\dfrac{N \, \{1\}}{2 \, \pi\,\sqrt{-1}}
\dfrac{d}{dm}
V_m(L)\right|_{\genfrac{}{}{0pt}{}{m=N}{q=\xi}}
.
\end{aligned}
$$
\label{th:main}
\end{theorem} 
\section{Relation to the hyperbolic volume}
In this section, we check the volume conjecture \eqref{eq:logconjecture}  of the logarithmic invariant $\gamma_s(K_{4_1})$ for the figure-eight knot $K_{4_1}$.  
\subsection{Logarithmic invariant of figure-eight knot}
The normalized colored Jones invariant $V_\lambda(K_{4_1})$ is expressed as follows.  
$$
V_s(K_{4_1})
=
\sum_{i=0}^\infty \frac{\{s+i, 2i+1\}}{\{1\}}.  
$$
This means that the coefficients $a_i(K_{4_1})$ are all equal to $1$ in Habiro's formula \eqref{eq:universal},   
and $\gamma_s^{(N)}(K_{4_1})$ is given by
\begin{equation}
\gamma_s^{(N)}(K_{4_1})
=
\sum_{i=0}^{\underline s-1} 
\{s+i, 2i+1\}
\sum_{k=s-i}^{s+i}
 \dfrac{ \{k\}_+}{\{k\}} 
 + 2\sum_{i={\underline s}}^{\overline s-1}
\widetilde{\{s+i, 2i+1\}},
\label{eq:gammas}
\end{equation}
where $\underline s = \min(s, \ N-s)$, 
$\overline s = \max(s, \ N-s)$ as before.  
%
%
\subsection{Limit of the logarithmic invariant}
For  $\gamma_s^{(N)}(K_{4_1})$, the following theorem holds.  
\begin{theorem}
Let $\alpha$ be a real number with $0 \leq \alpha< {\pi}/{3}$ and let 
$s_{N}^{\alpha}= \left\lfloor{N\alpha}/{2\,\pi} \right\rfloor$ where $\lfloor x\rfloor$ is the largest integer satisfying $\lfloor x \rfloor \leq x$.  
Then
\begin{equation}
\lim_{N\to\infty}
\frac{2 \, \pi \, \left|\gamma_{s_{N}^{\alpha}}^{(N)}(K_{4_1})\right|}{N}
=
\lim_{N\to\infty}
\frac{2 \, \pi \, \left|\gamma_{N-s_{N}^{\alpha}}^{(N)}(K_{4_1})\right|}{N}
=
\rm{Vol}\left(M_\alpha\right)
\label{eq:VC}
\end{equation}
where $M_\alpha$ is the cone manifold along singular set  $K_{4_1}$ with cone angle $\alpha$.
\label{th:volume}
\end{theorem}
\begin{remark}
Numerical computation suggests that
$$
\lim_{N\to\infty}
\frac{2 \pi \log\left|\gamma_{s_{N}^{\alpha}}^{(N)}(K_{4_1})\right|}{N}
=
\begin{cases}
\rm{Vol}(M_\alpha)  
&
\text{for $0 \leq \alpha < {2\pi}/{3}$},  
\\
0  &
\text{for ${2 \pi}/{3} \leq \alpha \leq {4 \pi}/{3}$}, 
\\
\rm{Vol}(M_{2\pi-\alpha})  
&
\text{for  
${4\pi}/{3} < \alpha \leq 2\pi$}.  
\end{cases}
$$
The values of ${2 \pi \log\left|\gamma_{s}^{(N)}(K_{4_1})\right|}/{N}$ for $N=200$ and $N=400$ are shown by graphs in Figure \ref{figure:graph}.  
\end{remark}
\begin{figure}[htb]
$$
\begin{matrix}
\begin{matrix}
\epsfig{file=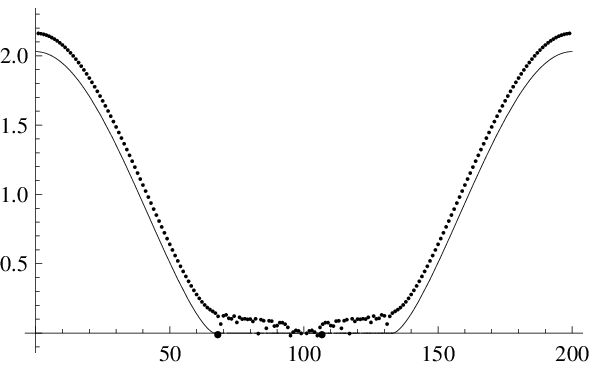, scale=1.1}
\raisebox{1mm}{\textcolor{black}{\ $s$}} &
\\[-3pt]\,\quad
\text{\small $0$}
\qquad\qquad\qquad\text{\small \ \ $\pi$}
\qquad\qquad\qquad\quad
\text{\small $2\pi$}{\, \raisebox{-1mm}{\textcolor{black}{$\alpha$}}}
\end{matrix}
&\!
\begin{matrix}
\epsfig{file=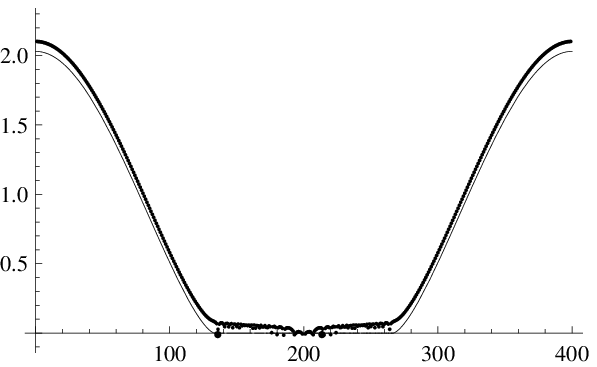, scale=1.1}
\raisebox{1mm}{\textcolor{black}{\ $s$}} &
\\[-3pt]\, \quad
\text{\small $0$}
\qquad\qquad\qquad\text{\small \ \ $\pi$}
\qquad\qquad\qquad\quad
\text{\small $2\pi$}{\, \raisebox{-1mm}{\textcolor{black}{$\alpha$}}}
\end{matrix}
\\
N=200 \qquad & N=400 \qquad
\end{matrix}
$$
\caption{Numerical computation for  
${2\pi\log\left|\gamma_s^{(N)}(K_{4_1})\right|}/{N}$ 
(dots) and the hyperbolic volume of the cone manifold $M_\alpha$ ($\alpha \leq \pi$), $M_{2\pi-\alpha}$ ($\alpha>\pi$)  along $K_{4_1}$ (thin line).}
\label{figure:graph}
\end{figure}
\par\noindent{\it Proof of Theorem \ref{th:volume}.}
We first prove 
$
\lim_{N\to\infty}
{2 \, \pi \, \left|\gamma_{s_{N}^{\alpha}}^{(N)}(K_{4_1})\right|}/{N}
=
\rm{Vol}\left(M_\alpha\right)
$
for $0 \leq \alpha < \pi/3$. 
To do this,   we start by estimating the sum 
$$
\sum_{i=0}^{s_{N}^{\alpha}-1} 
\{s_{N}^{\alpha}+i, 2i+1\}
\sum_{k=s_{N}^{\alpha}-i}^{s_{N}^{\alpha}+i} \frac{ \{k\}_+}{\{k\}}.
$$  
We know
$\left| { \{k\}_+}/{\{k\}}\right|
=
\left|\cot k \pi/N\right|
\leq 2
\left|\cot \pi/N\right|\leq N$
since $1 \leq k \leq N-1$,
and so $\left|\sum_{k=s_{N}^{\alpha}-i}^{s_{N}^{\alpha}+i} { \{k\}_+}/{\{k\}}\right| \leq N^2$.  
We also know $\{s_{N}^{\alpha}+i\}\{s_{N}^{\alpha}-i\} \leq 1$ since $0 \leq s_{N}^{\alpha} \leq {N}/{6}$ and $0 \leq i \leq s_{N}^{\alpha}$.  
Therefore, we have
$$
\left|\sum_{i=0}^{s_{N}^{\alpha}-1} 
\{s_{N}^{\alpha}+i, 2i+1\}
\sum_{k=s_{N}^{\alpha}-i}^{s_{N}^{\alpha}+i} \frac{ \{k\}_+}{\{k\}}
\right|
\leq
N^3.  
$$
Next we estimate 
$\sum_{i={s_{N}^{\alpha}}}^{N- s_{N}^{\alpha}-1}
\widetilde{\{s_{N}^{\alpha}+i, 2i+1\}}$.  
Let $a_i = (-1)^{s_{N}^{\alpha}-1} \widetilde{ \{s_{N}^{\alpha}+i, 2i+1\}}$.  
Then $a_i\geq 0$ and we have
$$
a_{i_{\max}^{(N)}}
\leq 
\sum_{i= s_{N}^{\alpha} + 1}^{N- s_{N}^{\alpha}} a_k
\leq N \, a_{i_{\max}^{(N)}}
$$
where $a_{i_{\max}^{(N)}} = \max_{s_{N}^{\alpha}  \leq i \leq N- s_{N}^{\alpha}-1}^{}(a_i)$.  
Therefore, 
\begin{equation}
- N^3 + N \, a_{i_{\max}^{(N)}} \leq
\left|\gamma_{s_{N}^{\alpha}}(K_{4_1})\right| 
\leq N^3 + N \, a_{i_{\max}^{(N)}}
\label{eq:proof}
\end{equation}
The index $i_{\max}^{(N)}$ for the maximum $a_i$ must be equal to   $i_1^{(N)} =  N-s_{N}^{\alpha}-1$ or $i_2^{(N)}$ satisfying $\{s_{N}^{\alpha}+i_2^{(N)}\}\{s_{N}^{\alpha}-i_2^{(N)}\} \geq 1$ and $\{s_{N}^{\alpha}+i_2^{(N)}+1\}\{s_{N}^{\alpha}-i_2^{(N)}-1\} \leq 1$
since $i_1^{(N)}$ and $i_2{(N)}$ correspond to the local maximal of $a_i^{}$.  
The index $i_2^{(N)}$ satisfies
\begin{equation}
\cos \frac{2\, \pi\, (i_{2}^{(N)}+1)}{N}  \leq \cos\frac{2\, \pi \, s_{N, \alpha}^{}}{N} - \frac{1}{2}
\leq \cos \frac{2\, \pi\, i_{2}^{(N)}}{N}.  
\label{eq:i0} 
\end{equation}
If $N$ is not small, such $i_2^{(N)}$ exists uniquely between ${N}/{2}+s_{N}^{\alpha}$ and $N-s_{N}^{\alpha}-1$ because of  
$\{{N}/{2}+2\, s_{N}^{\alpha}\}\{- {N}/{2}\} >1$, 
$\{2\,s_{N}^{\alpha}+1-N\}\{N-1\} < 1$
and the shape of the graph of the cosine function.  
The $\log$ of $a_i$ is given by
$$
\log a_{i}
=
\sum_{k=s_{N}^{\alpha}-i}^{-1}\log|\{k\}|
+
\sum_{k=1}^{s_{N}^{\alpha}+i}
\log|\{k\}|
$$ 
and is estimated as
\begin{multline*}
N
\int_{\frac{s_{N}^{\alpha} - i}{N}}^{0}
\log\left|2 \sin {\pi t}\right|\, dt
+
N
\int_{0}^{\frac{s_{N}^{\alpha} + i}{N}}
\log\left|2 \sin {\pi t}\right|\, dt
<
\log a_{i}
<
\\
N
\int_{\frac{s_{N}^{\alpha} - i+1}{N}}^{-\frac{1}{N}}
\log\left|2 \sin {\pi t}\right|\, dt
+
N
\int_{\frac{1}{N}}^{\frac{s_{N}^{\alpha} + i+1}{N}}
\log\left|2 \sin {\pi t}\right|\, dt.  
\end{multline*}
Therefore,
$$
\lim_{N\to\infty}\!\frac{2 \pi \log a_{i_1^{(N)}}}{N}
= 
-2 \,\Lambda(\alpha), 
\qquad
\lim_{N\to\infty}\!\frac{2  \pi \log a_{i_2^{(N)}}}{N}
= 
-2 \left(\Lambda( \frac{\alpha+ \theta}{2}) - \Lambda(\frac{\alpha-\theta}{2})
\right)
$$
where $ \theta  = \lim_{N\to\infty}{2\pi i_2^{(N)}}/{N}$ and $\Lambda(x) = -\int_0^x \log \left|2\, \sin t\right|\,dt$ is the Lobachevski function.  
Then $\theta > \pi$ since 
${N}/{2}+s_{N}^{\alpha} < i_2^{(N)} < N-s_{N}^{\alpha}-1$,  and this implies that
$\Lambda\big((\alpha-\theta)/2\big) > \Lambda\big((\alpha+\theta)/2\big)$.  
We also know that $\Lambda(2\, \alpha)>0$.  
Therefor, $\lim_{N\to\infty}a_{i_1^{(N)}}<1$,  
$\lim_{N\to\infty}a_{i_2^{(N)}}>1$, and we have 
$i_{\max}^{(N)} = i_2^{(N)}$ for sufficient large $N$.  
By using \eqref{eq:proof} and the fact that $\lim_{N\to\infty}{\log N}/{N} = 0$, we get
$$
\lim_{N\to\infty}\frac{2 \, \pi \, \log\left|\gamma_{s_{N}^{\alpha}}^{(N)}(K_{4_1})\right|}{N}
= 
-2 \left(
\Lambda(\frac{\alpha+\theta}{2}) - \Lambda(\frac{\alpha-\theta}{2})\right),
$$
where $\theta$ satisfies
$
\cos \theta =\cos \alpha - {1}/{2}  
$
by \eqref{eq:i0}.    
The right hand side of this formula is equal to the hyperbolic volume of the cone manifold $M_\alpha$ given by Mednykh \cite{Me} since
\begin{multline*}
\frac{d}{d\alpha}\left( -2 \left(
\Lambda(\frac{\alpha+\theta}{2}) - \Lambda(\frac{\alpha-\theta}{2})\right)
\right) = 
\log \left|t - \sqrt{t^2-1}\right|
=
-\operatorname{arccosh} t
= \frac{d}{d\alpha} \rm{Vol}(M_\alpha),
\end{multline*}  
where $t = 1 + \cos \alpha - \cos 2 \alpha$, and, 
if $\alpha=0$, then $\theta =  \pi/3$ and 
$-2\big(\Lambda(\pi/6)-\Lambda(-\pi/6)\big) = 
\rm{Vol}(S^3 \setminus K_{4_1}) = \rm{Vol}(M_0)$.  
\par
The proof for another equality is similar.  
\qed
%
%
%
%
%
%

%


\begin{thebibliography}{99}
%
%
\bibitem{F}
B. L. Feigin, A. M. Gainutdinov,  A. M. Semikhatov, and I. Yu. Tipunin, 
{\it Modular group representations and fusion in logarithmic conformal field theories and in the quantum group center,} 
Commun. Math. Phys. {\bf 265} (2006), 47--93. 
%
\bibitem{G}
S. Gukov,
{\it Three-dimensional quantum gravity, Chern-Simons theory, and the A-polynomial,} 
Commun. Math. Phys. {\bf 255} (2005), 577--627. 
%
\bibitem{H}
K. Habiro, 
{\it On the quantum $sl_2$ invariants of knots and integral homology spheres,}
Geometry and Topology Monographs {4}, Invariants of knots and 3-manifolds (Kyoto 2001), 55--68, Geom. Topol. Publ., Coventry, 2002  
%
\bibitem{K}
R. M. Kashaev, 
{\it The hyperbolic volume of knots from the quantum dilogarithm,}
Lett. Math. Phys. {\bf 39}  (1997), 269--275.
%
%
\bibitem{L}
R. J. Lawrence,  
{\it A universal link invariant using quantum groups,}
Differential geometric methods in theoretical physics (Chester, 1988), 55--63, World Sci. Publ., Teaneck, NJ, 1989.
%
\bibitem{Ma}
G. Masbaum,
{\it Skein-theoretical derivation of some formulas of Habiro,}
Algebr. Geom. Topol. {\bf 3} (2003), 537--556.
%
\bibitem{Me}
A. D. Mednykh, 
{\it On hyperbolic and spherical volumes for knot and link cone-manifolds,} 
Kleinian Groups and Hyperbolic 3-Manifolds (Warwick, 2001).
London Math. Soc. Lecture Note Ser. {\bf 299}, 
Cambridge Univ. Press, Cambridge, 145--163, 2003.
%
\bibitem{HM}
H. Murakami, 
{\it Some limits of the colored Jones polynomials of the figure-eight knot,} 
Kyungpook Math. J. {\bf 44} (2004),  369--383.
%
\bibitem{MM}
H. Murakami, and J. Murakami, 
{\it The colored Jones polynomials and the simplicial volume of a knot,}
Acta Math. {\bf 186} (2001), 85--104. 
%
\bibitem{MMOTY}
H. Murakami, J. Murakami, M. Okamoto, T. Takata, and Y. Yokota, 
{\it Kashaev's conjecture and the Chern-Simons invariants of knots and links,}
Experiment. Math. {\bf 11} (2002), 42--435.
%
\bibitem{MY}
H. Murakami, and Y. Yokota, 
{\it The colored Jones polynomials of the figure-eight knot and its Dehn surgery spaces,}
J.\ Reine\ Angew.\ Math. {\bf 607} (2007), 47--68.
%
%
%
\bibitem{MN}
J. Murakami, and K.  Nagatomo, 
{\it Logarithmic knot invariants arising from restricted quantum groups,}
Internat.\ J.\ Math. {\bf 19} (2008), 1203--1213.
%
\bibitem{O}
T. Ohtsuki, 
{\it Colored ribbon Hopf algebras and universal invariants of framed links.}
J. Knot Theory Ramifications {\bf 2} (1993),  211--232.
%
\bibitem{RT}
N. Reshetikhin, and V. G. Turaev,
{\it Ribbon graphs and their invariants derived from quantum groups,}
Comm. Math. Phys. {\bf 127} (1990), 1--26.  
%
%
\end{thebibliography}
\end{document}